\documentclass[10pt]{article}

\RequirePackage{amsthm,amsmath,amsfonts,amssymb}
\RequirePackage[numbers]{natbib}

\usepackage{amsthm,amsmath,amssymb,amsfonts,bbm}
\usepackage{tikz}
\usepackage{dsfont}
\usepackage{centernot}
\usepackage{graphicx}
\usepackage[abs]{overpic}
\usepackage{xcolor,varwidth}

\usepackage{ stmaryrd }

\newtheorem{lemma}{Lemma}[section]
\newtheorem{proposition}{Proposition}[section]
\newtheorem{corollary}{Corollary}[section]
\newtheorem{remark}{Remark}[section]
\newtheorem{claim}{Claim}[section]
\newtheorem{theorem}{Theorem}

\newtheorem{definition}{\emph{Definition}}[section]

\newcommand{\pr}{\mathbb{P}}

\newcommand{\Z}{ \mathbb{Z}}

\numberwithin{equation}{section}

\definecolor{miverde}{RGB}{66,178,66} 

\title{Convergence of the  one-dimensional contact process with two types of particles and priority}
\author{Mariela Pentón Machado}
\begin{document}
\maketitle
\begin{abstract}
We consider a symmetric finite-range contact process on $\mathbb{Z}$ with two types of particles (or infections), which propagate according to 
the same supercritical rate and die (or heal) at rate $1$.    Particles  of type $1$ can enter any site in $(-\infty ,0]$ that is empty or occupied by a particle of type $2$ and, analogously, particles of type $2$ can enter any site in $[1,\infty)$ that is empty or occupied by a particle of type $1$. Also, almost one particle can occupy each site. We prove that the process with initial configuration  $\mathds{1}_{(-\infty,0]}+2\mathds{1}_{[1,\infty)}$ converges in distribution to an invariant measure different from the non trivial invariant measure of the classic contact process.  In addition, we prove that for any initial configuration the process converges to a convex combination of four invariant measures.
\end{abstract}


\section{Introduction}
In this work, we study the set of invariant measures of the contact process with two types of particles and priority. This process is a stochastic process that can be interpreted as the temporal evolution of a population that has two different species and each of them has a favorable region in the environment.

The \emph{classic contact process} was introduced in \cite{Harris} and is a process widely studied in the literature. In this process, every infected individual can propagate the infection  at rate $\lambda$ to some neighbor at distance $R$ and it becomes healthy at rate $1$. This process also can be interpreted as the time evolution of a certain population, where a site is now ``occupied" (in correspondence to ``infected") or ``empty" (in correspondence to ``healthy"). The classic contact process presents a dynamical phase transition, namely:  there exists a critical value $\lambda_c$ for the infection rate such that if $\lambda$ is larger than $\lambda_c$, there is a non-trivial invariant measure $\mu$ different from $\delta_{\emptyset}$. 

 The contact process with two types of particles and priority is a  continuous-time Markov process $\{\zeta_t\}_{t\geq 0}$ on $\{0,1,2\}^\mathbb{Z}$. If $\zeta_t(x)=i$, then the site $x$ is occupied at time $t$ by a particle of type $i$ ($i=1,2$) and if $\zeta_t(x)=0$, then the site $x$ is empty at time $t$. 
  We denote  the flip rates at  $x$ in a configuration $\zeta \in \{0,1,2\}^{\mathbb{Z}}$ by $c(x,\zeta, \cdot)$ and these are defined as follows
 $$ \begin{array}{ll}
 c(x, \zeta, 1\rightarrow 0 )= c(x, \zeta,2\rightarrow 0)= 1,\\
 c(x, \zeta, 0 \rightarrow i)=   \lambda \underset{y: \hspace{0.1cm} 0<||x-y||\leq R }{\sum} \mathds{1}_{\zeta(y)=i} , i= 1,2,\\ 
  c(x, \zeta, 2 \rightarrow 1)=   \lambda \underset{y: \hspace{0.1cm}  0<||x-y||\leq R }{\sum} \mathds{1}_{\zeta(y)=1}\mathds{1}_{\{x \in(-\infty,0]\}} ,\\
  c(x, \zeta, 1 \rightarrow 2)=   \lambda  \underset{y: \hspace{0.1cm}   0<||x-y||\leq R }{\sum} \mathds{1}_{\zeta(y)=2}\mathds{1}_{\{x \in [1,\infty)\}}.
\end{array}
$$
The above flip rates give the following rules for the
dynamics:
\begin{itemize}
    \item  a site occupied by a particle of type $i$ becomes  empty with rate $1$;
    \item  a  particle of type $i$   gives birth to a particle of type $i$ at sites within the range $R$ with rate $\lambda$, but
    \item  type $ 1 $  particles cannot occupy places occupied by  type $ 2 $ particles in $ [1, \infty) $ and, vice versa,  type $ 2 $ particles cannot occupy places occupied by  type $ 1 $ particles in $ (- \infty, 0] $.
\end{itemize}
We consider $ R\geq 1 $ and restrict the process to the supercritical case, where $ \lambda> \lambda_ {c} = \lambda_ {c} (R) $. This process can be interpreted as the time evolution of a population in which there are two types of individuals, type $ 1 $ and type $ 2 $. Each type of individual has a priority zone, type $ 1 $ has priority in $ (- \infty, 0] $ and type $ 2 $ in $ [1, \infty) $. This model is inspired by the Grass-Bushes-Trees model, introduced in \cite{GBT}, in this case type $ 1 $ individuals have priority throughout the environment.

  We denote by $\mu_1$ (resp. $\mu_2$) the measure in $\{0,1,2\}^{\mathbb{Z}}$ supported on the configurations without particles of type $2$ (resp.  type $1$), such that this measure restricted to $\{0,1\}^{\mathbb{Z}}$ (resp. $\{0,2\}^{\mathbb{Z}}$) is the non-trivial invariant measure for the classic contact process, $\mu$.   Note that if the initial configuration only has one type of particle, the process with two types of particles is the same as the classic contact process. Therefore,  $\mu_1$ and $\mu_2$ are both invariant measures for the contact process with two types of particles. In the first theorem of this paper, we prove that, starting with the initial configuration
$ \mathds{1}_{(- \infty, 0]} + 2 \mathds{1}_{[1, \infty)} $, the contact process with two types of particles converges to an invariant measure $ \nu $ , which is different from $ \mu_1 $ and $ \mu_2 $. In our second result, we show  that for any initial configuration, the contact process with two types of particles converges to a convex combination of the four measures $ \delta_{\emptyset} $, $ \mu_1 $, $ \mu_2 $ and $ \nu $.

The paper is organized as follows. In Section \ref{Preliminaries}, we introduce the notation and state our main two results. In Section \ref{Moreresultskdenpendent}, we recall tools from oriented percolation and the Mountford-Sweet renormalization introduced in \cite{MountfordSweet}. In Section \ref{Convergence results}, we prove our first main result.  In Subsection \ref{Proof of invariant measure}, we prove the existence of the invariant measure $\nu$ supported in the set of configurations in $\{0,1,2\}^{\mathbb{Z}}$ with infinitely many particles of type $1$ and infinitely many particles of type $2$. Also, in this subsection, we create all the tools to finally prove Theorem $1$ in Subsection \ref{Proof of main Theorem}. In Section \ref{Proof theo 2}, we prove Theorem $2$.
\section{Preliminaries and statements of the main results}\label{Preliminaries}
\textbf{Notations} We denote by $||\cdot||$ the euclidean norm in $\mathbb{R}$ and we use $|\cdot|$ for the cardinality of subsets in $\mathbb{R}$ and $\mathbb{R}^{2}$.  During all the work, we refer to the contact process with two types of particles and priority as the \emph{two-type contact process} and the process with only one type of particle as the classic contact process.    For the initial configuration $\mathds{1}_{(-\infty,0]}+2\mathds{1}_{[1,\infty)}$, we denote the two-type contact process  by  $\zeta^{\textbf{1},\textbf{2}}_{t}$.
 We stress that, during the paper, the letter $\xi$ refers to the classic contact process, and $\zeta$ refers to the two-type contact process.  To simplify the notation,  throughout  the paper we identify  $I\cap\mathbb{Z}$ with $I$ for every spatial interval. Also, we identify every configuration $\xi$ in $\{0,1\}^{\mathbb{Z}}$ with the subset $\{x\in \mathbb{Z}: \xi(x)=1\}$.
In addition, we identify every $\zeta$ in $\{0,1,2\}^{\mathbb{Z}}$ with the disjoint subsets  $A=\{x\in \mathbb{Z}: \zeta(x)=1\}$ and $B=\{y\in \mathbb{Z}:\zeta(y)=2\}$. 

\noindent\textbf{The classic contact process.} To define the classic contact process with range $R\in \mathbb{N}$ and rate of infection $\lambda>0$, we consider a collection of independent Poisson point processes (PPP)  on $[0,\infty)$
\begin{align*}
&\{P^{x}\}_{x \in \mathbb{Z}}\text{ with rate 1},\quad\{P^{x\rightarrow y}\}_{\{x,y \in \mathbb{Z}: \hspace{0.2cm} 0<||x-y||\leq R\}} \text{ with rate }\lambda.
\end{align*}
All these processes are defined on a probability space $(\Omega, \mathcal{F}, \pr)$. Graphically, we  place a cross mark at the point $(x,t)\in \mathbb{Z}\times [0, +\infty)$
whenever $t$ belongs to the Poisson process $P^x$. In addition, we place an arrow following the direction from  $x$ to $y$  whenever $t$ belongs to the Poisson process $P^{x\rightarrow y}$. We  denote by $\mathcal{H}$  the collection of these marks in $\mathbb{Z}\times[0,\infty)$, this is a \emph{Harris construction} (see Figure \ref{HarrisConstrution}). We denote by $\mathcal{F}_t$ the $\sigma$-algebra generated by the collection of PPP until time $t$. 
 \begin{figure}[h]
	\centering
		  \includegraphics[scale=0.2,unit=1mm]{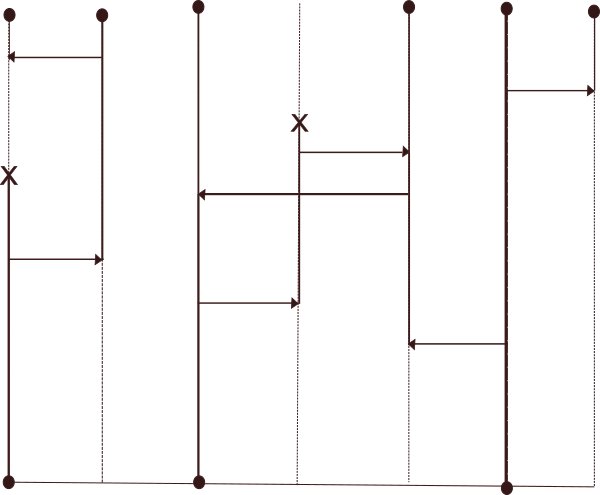}
   \caption{An example of a Harris construction for $R>1$.}\label{HarrisConstrution}
   \end{figure}

A \emph{path} in $\mathcal{H}$ is an oriented path that follows the  positive direction of time $t$, passes along the arrows in the direction of them and does not pass through any cross mark. More precisely,  a path from $(x,s)$ to $(y,t)$, with $0<s<t$, is a piecewise constant function $\gamma:[s,t]\rightarrow \mathbb{Z}$ such that:
\begin{itemize}
 \item{$ \gamma(s)=x, \gamma(t)=y$},
 \item{$\gamma(r)\neq \gamma(r-)$ only if  $r \in P^{\gamma(r-)\rightarrow \gamma(r)}$}\footnote{The notation $r\in P^{x\rightarrow y}$ means that $r\in (0,\infty)$ is a jump time of the Poisson process $P^{x\rightarrow y}$. },
 \item{$\forall r \in [s,t], r \notin  P^{\gamma(r)}$.}
\end{itemize}
In this case, we say that $\gamma$ connects $(x,s)$ and $(y,t)$. Moreover, if such a path exists, we write $(x,s)\rightarrow(y,t)$.

 For  $A$, $B$  and  $C$ subsets of $\mathbb{Z}$ and $0\leq s<t$, we say that $A \times \{s \}$ is connected with $B\times\{t \} $ inside $C$, if there exist $x\in A$, $y \in B$ and a path $\gamma$ connecting $(x,s)$ with $(y,t)$ such that $\gamma( r)\in C$ for all $r$, $s\leq r\leq t$. We denote this situation by $A\times \{s\}\rightarrow B\times \{t\}$ \emph{inside} $C$.

 Given a Harris construction $\mathcal{H}$ and  a subset $A$   of  $\mathbb{Z}$, we define the classic contact process beginning at time $s$  with initial configuration $A$ as follows
\begin{equation*}
\xi^{A}_{s,t}=\{x: \text{ exists } y \in A \text{ such that }(y,s)\rightarrow(x,t)\}.
\end{equation*}
In the special case of $s=0$, we just write $\xi^{A}_t$.  Also,  we denote by $\xi^{x}_{t}$ the process with initial configuration $\{x\}$. Furthermore,  we define the time of extinction of $ \xi^{A}_t$  as follows
\begin{equation*}
T^{A}=\inf\{t>0: \xi^{A}_t=\emptyset\}.
\end{equation*}
 By the graphic construction, we have the Markov property for the classical contact process. 
 
 For a time $t$ and  a set $A$,  we define the dual contact process  at time $s\in[0,t]$, with initial configuration $A$, by 
\begin{equation*}
\tilde{\xi}^{A,t}(s)=\{x: \text{ there exists }y \in A \text{ such that }(x,t-s)\rightarrow(y,t)\}.
\end{equation*}
We observe that the process $\{\tilde{\xi}^{A,t}(s)\}_{0\leq s\leq t}$ has the same law as the classic contact process  at time $t$ with initial configuration $A$. 
 
As we mentioned in the introduction, the classic contact process presents a phase transition in the rate of infection $\lambda$: there exists a critical parameter $\lambda_c=\lambda_c(R)$  defined  as follows
$$
\lambda_{c}=\inf\{\lambda:\pr(T^{0}(\lambda)=\infty)>0\}.
$$
For all $\lambda>\lambda_{c}$ all invariant measures of the process are a convex combination of  $\delta_{\emptyset}$ and a non trivial measure $\mu=\mu(\lambda)$. During all our work we are considering $\lambda>\lambda_c$. 

For the contact process with initial configuration $(-\infty,0]$ we define the rightmost occupied site at time $t$ as 
$$
r^{(-\infty,0]}_t=\max\{x:\xi^{(-\infty,0]}_t(x)=1\}.
$$
It is well known that there exists a positive number $\alpha$ such that
\begin{equation}\label{defalpha}
\lim\limits_{t\rightarrow \infty}\frac{r^{(-\infty,0]}_t}{t}=\alpha\text{ almost surely.}
\end{equation}
The following result is a simple lemma, which will be used in the next sections.
\begin{lemma} Let $C$ be a subset of $\mathbb{Z}$, then we have that 
\begin{equation}\label{sextaequacaoahi}
\pr(\nexists\, s: |\xi^{C}_s|\geq N;T^{C}=\infty)=0\quad   \forall N\in \mathbb{N}.
\end{equation}
\end{lemma}
\begin{proof}
Let $N$ be a positive integer. To obtain \eqref{sextaequacaoahi},  we first observe that
\begin{small}
\begin{align}
\begin{split}\label{setimaequacaoahi}
     &\pr(\forall\, 1\leq k\leq n\,|\xi^{C}_k|\leq N;\xi^{C}_n\neq \emptyset)=\int\limits_{|\xi|\leq N;\xi\neq \emptyset}\pr(|\xi^{C}_k|\leq N;\,1\leq k\leq n;\xi^{C}_n\neq \emptyset|\xi^{C}_{n-1}=\xi)d\mu_{n-1}(\xi)\\
    &=\int\limits_{|\xi|\leq N;\xi\neq \emptyset}\pr(|\xi^{\xi}_{1}|\leq N;\xi^{\xi}_1\neq \emptyset)\pr(|\xi^{C}_k|\leq N;\,1\leq k\leq n-2|\xi^{C}_{n-1}=\xi)d\mu_{n-1}(\xi)\\
    &\leq\int\limits_{|\xi|\leq N;\xi\neq \emptyset}\pr(\xi^{\xi}_1\neq \emptyset)\pr(|\xi^{C}_k|\leq N;\,1\leq k\leq n-2|\xi^{C}_{n-1}=\xi)d\mu_{n-1}(\xi),
\end{split}
\end{align}
\end{small}where $\mu_{n-1}$ is the distribution of $\xi^{C}_{n-1}$. We observe that  in the second equality of  \eqref{setimaequacaoahi}, we have used the Markov property of the contact process. Moreover, we have that  $(1-e^{-1})^Ne^{-2RN\lambda}$ is the probability that before time $1$,  there are no marks for $2RN$ independent Poisson processes  of rate $\lambda$, and there is a mark before time $1$ for $N$ independent Poisson processes of rate $1$. Therefore, in the last term of \eqref{setimaequacaoahi}, we have that  the  first probability within the integral is less than $1-(1-e^{-1})^Ne^{-2RN\lambda}$. Then, we conclude that for any $n\geq 0$ 
\begin{align}\label{oitavaequacaoahi}
\begin{split}
    &\pr(\forall\, 1\leq k\leq n\,|\xi^{C}_k|\leq N;\xi^{C}_n\neq \emptyset)\\
&\qquad\leq \left(1-(1-e^{-1})^Ne^{-2RN\lambda}\right)\pr(|\xi^{C}_k|\leq N;\,1\leq k\leq n-1;\xi^{C}_{n-1}\neq \emptyset).
\end{split}
\end{align}
   Using \eqref{oitavaequacaoahi} recursively in $n$, we obtain 
   $$
  \pr(\forall\, 1\leq k\leq n\,|\xi^{C}_k|\leq N;\xi^{C}_n\neq \emptyset)\leq\left(1-(1-e^{-1})^Ne^{-2RN\lambda}\right)^{n},
   $$
  for all $n$ and \eqref{sextaequacaoahi} follows.
\end{proof}
\noindent \textbf{The two-type contact process.} We now define the two-type contact process using the  Harris construction. The advantage of this definition is that it provides a coupling between the classic contact process and the  two-type contact process.
    
 First, we define the two-type contact process restricted to the interval $[-N+1,N]$. 
 Let  $A$ and $B$ be two disjoint subsets of $[-N+1,N]$, we denote by $\{\zeta^{A,B,N}_t\}_t$ the  two-type contact process restricted to $[-N+1,N]$ with initial configuration $\mathds{1}_{A}+2\mathds{1}_{B}$. In this case, it is simple to define this process in terms of  a Harris construction, since we are dealing with a stochastic process that has c\`adl\`ag trajectories with jumps only in the times of the Poisson processes $\{P^{x}\}_{x\in [-N+1,N ]}$ or $\{P^{y\rightarrow x}\}_{\{y,x \in [-N+1,N ]: \hspace{0.2cm} 0<||x-y||\leq R\}}$. Let  $t$ be one of those times, two scenarios are possible:\begin{itemize}
    \item[(1)]  $t\in P^{x} $ for some $x$. In this case, $x$ is empty at this time and we set $\zeta^{A,B,N}_t(x)=0$;
    \item[(2)] $t\in P^{y\rightarrow x}$ for some $x$ and $y$.  If $x$ is occupied by a particle of type $i$ ($i=1,2$), and $x$ is in the region of priority of this type of particles, then nothing changes at $x$. Otherwise, $x$ becomes occupied by the type of particle that is in $y$ and we set $\zeta^{A,B,N}_t(x)=\zeta^{A,B,N}_t(y)$.
\end{itemize}

Now, let $A$ and $B$ be two disjoint subsets of $\mathbb{Z}$ and $q$ a positive rational number. In the set $\Omega_q=\underset{x\in \mathbb{Z}}{\cap}\{|\tilde{\xi}^{\{x\},q}(q)|<\infty\}$, which has probability one, we define the  two-type contact process with  initial configuration $\mathds{1}_{A}+2\mathds{1}_{B}$ at time $q$ as
$$
\zeta^{A,B}_q(x)=\lim\limits_{N\rightarrow\infty}\zeta^{A_N,B_N,N}_q(x)
$$
for every $x\in \mathbb{Z}$, where $A_N=A\cap[-N+1,N]$ and $B_N=B\cap[-N+1,N]$. Moreover, in the set $\tilde{\Omega}=\underset{q\in \mathbb{Q}^+}{\cap}\Omega_q$, which also have total probability, we define  the two-type contact process with initial configuration $\mathds{1}_{A}+2\mathds{1}_{B}$ at time $t$ as
$$
\zeta^{A,B}_t=\lim\limits_{q\downarrow t}\zeta^{A,B}_q,
$$
for every $t\geq 0$. In this way, we have defined a stochastic process with c\`adl\`ag trajectories and with flip rates as described in the introduction.
We also observe that, as for the classic contact process, the Markov property holds for the  two-type contact process.
  
For a configuration $\zeta\in\{0,1,2\}^{\mathbb{Z}}$  we define the rightmost site occupied by a type $1$ particle as 
\begin{equation*}
\textbf{r}^1(\zeta)=\sup\{x: \zeta(x)=1\},
\end{equation*}
 and the leftmost site occupied by a type $2$ particle as
\begin{equation*}
\textbf{l}^2(\zeta)=\inf\{x:\zeta(x)=2\}
\end{equation*}
with the convention that $\sup\{\emptyset\}=-\infty$ and $\inf\{\emptyset\}=\infty$.

Now, we are ready to state the two main results of our paper.

\begin{theorem}\label{teoconvergence} There exists an invariant measure $\nu$ for the two-type contact process such that 
\begin{equation*}
\zeta^{\textbf{1},\textbf{2}}_t\underset{t\rightarrow\infty}{\longrightarrow}\nu \text{ in Distribution.}
\end{equation*}
\end{theorem}

\begin{theorem}\label{Teo2}
Let $A$ and $B$ be two disjoint subsets of $\mathbb{Z}$. The process $\{\zeta^{A,B}_t\}$ converges to a convex combination of the measures $\delta_{\emptyset}$, $\mu_1$, $\mu_2$ and $\nu$. Consequently, the set of stationary and  extremal distributions for the two-type contact process  is  $\{\delta_{\emptyset},\mu_1, \mu_2,\nu\}$.
\end{theorem}
\section{$k$-dependent percolation systems with small closure and the Mountford-Sweet renormalization}\label{Moreresultskdenpendent}

In this section, we first introduce some notations and results for oriented percolation. After these notions, we recall the definition of the Mountford-Sweet renormalization for the contact process with $R>1$. 

Consider $\Lambda=\{(m,n)\in \mathbb{Z}\times \mathbb{Z}^{+}: m+n \text{ is even}\}$, $X=\{0,1\}^{\Lambda}$ and $\mathcal{X}$ the $\sigma$-algebra generated by the cylinder sets of $X$. 
  Given $\Psi \in X$, we say that two points $(m,k)$, $(m',k')\in\Lambda$ with $k<k'$ are \emph{connected by an open path} (\emph{according to} $\Psi$) \cite{Conos}, if there is a sequence $\{(m_i, n_i)\}_{0\leq i\leq k'-k}$  such that 
 $$(m_0,n_0)=(m, k),\quad (m_{k'-k},n_{k'-k})=(m', k'),\quad ||m_{i+1}-m_{i}||=1,\quad  n_i=k+i,$$ with $0\leq i\leq k'-k-1$ and $\Psi(m_i, n_i)=1$ for all $i$. 
 If $(m,k)$ and $(m',k')$ are connected by an open path (according to $\Psi$), we write $(m,k)\rightsquigarrow (m',k')$ (according to $\Psi$).

Now, let $A$ and $B$ be subsets of $\mathbb{Z}$ and $C$ be a subset of $\Lambda$. We say that $A \times \{n \}$ is connected with $B\times\{n' \} $  inside $C$, if there are $m\in A$ and $m' \in B$ such that $(m,n)$, $(m',n')$ are in $\Lambda$, $(m,n)\rightsquigarrow (m',n')$ and all the edges of the path are in $C$. In this case, we write $A\times\{n\}\rightsquigarrow B\times\{n'\} \text{ inside }C$.
Let $(m,n)$ be a point in $\Lambda$, we define the cluster beginning at $(m,n)$ as 
\begin{equation*}
C_{(m,n)}=\{(m',n'):\text{ such that }n'\geq n \text{ and }(m,n)\rightsquigarrow(m',n')\}.
\end{equation*}
 
 Let $A$ be a subset of $\mathbb{Z}$ such that $\sup A <\infty$, we define the rightmost site connected with $A\times\{0\}$  at time $n$  as follows
\begin{equation*}
\hat{r}^{A}_n=\max\{k:\exists \,(k',0)\in A\times\{0\},\,\text{ such that }(k',0)\rightsquigarrow(k,n)\}.
\end{equation*}
Let $B$ be a subset of $\mathbb{Z}$ such that $\inf B>-\infty$, we define the leftmost site connected with $B\times\{0\}$ at time $n$ as
\begin{equation*}
\hat{l}^{B}_n=\min\{k:\exists \,(k',0)\in B\times\{0\},\,\text{ such that }(k',0)\rightsquigarrow(k,n)\}.
\end{equation*}

Let $\hat{\mathbb{P}}_{p}=\prod_{\Lambda} (p\delta_1+(1-p)\delta_0)$ be the Bernoulli product measure on $\Lambda$.  In \cite{Survey}, it was proved via the dual-contours methods that
\begin{equation}\label{limcluster}
\underset{p \rightarrow 1}{\lim}\, \hat{\pr}_{p}(|C_{(0,0)}|=\infty)=1,
\end{equation}
and for every $\beta \in (0,1)$
\begin{equation}\label{percolation slope}
\underset{p \rightarrow 1}{\lim}\,\hat{\pr}_{p}(\exists\hspace{0.2 cm} n \geq 1: \hat{r}^{(-\infty,0]}_n < \beta n)=0.
\end{equation}
Given $k\geq 1$ and $\delta>0$,  $(X,\mathcal{X}, \hat{\mathbb{P}})$ is a \emph{$k$-dependent oriented percolation system with closure below $\delta$}, if  for all $r$ positive
\begin{small}
$$
\hat{\mathbb{P}}( \Psi(m_i,n)=0, \forall i\hspace{ 0.2 cm} 0\leq i\leq r|\{\Psi(m,s): (m,s)\in \Lambda, 0\leq s<n\} )<\delta^r,
$$
\end{small}with $(m_i,n)\in \Lambda$ and $||m_i-m_j||>2k$ for all $i \neq j$ and $1 \leq i,j \leq r$ (see \cite{MountfordSweet}, \cite{Conos}).

Let $\Psi$ and $\Psi'$ be two elements of $X$, we say that  $\Psi\leq \Psi'$  if $\Psi(m,n)\leq\Psi'(m,n)$ for all $(m,n)\in \Lambda$. Also,  we say that a subset  $A$ of $X$ is increasing if $\Psi \in A$ and $\Psi\leq\Psi'$,  then $\Psi' \in A$.  Let $\hat{\mathbb{P}}_1$ and $\hat{\mathbb{P}}_2$ be two probability measures on $\mathcal{X}$, we say that $\hat{\mathbb{P}}_1$ stochastically dominates $\hat{\mathbb{P}}_2$ if $\hat{\mathbb{P}}_1(A)\geq\hat{\mathbb{P}}_2(A)$ for all $A$ increasing in $\mathcal{X}$.

The following lemma is a consequence of Theorem $0.0$ in \cite{ LiggetSchonmannStacey}.
\begin{lemma}\label{LiggettSmallClosure}
For  $k\in \mathbb{N}$  and  $0<p<1$ fixed, there exists $\delta>0$ such that if $(X,\mathcal{X},\hat{\pr})$ is a  k-dependent oriented percolation system with closure below   $\delta$,  then    $\hat{\mathbb{P}}$  stochastically dominates $\hat{\mathbb{P}}_{p}$.
\end{lemma}
In the next two lemmas, we enunciate basic results for the Bernoulli product measure and some consequences of  these results for $k$-dependent percolation systems with small closure.  Items $(i)$ and $(ii)$ in Lemma \ref{Lemmak-dependentpercolation} below can be found in \cite{Conos}. Before the statements of the lemmas, we define the following sets
$$
A_{n}(i,k)=\{ \exists\hspace{0,1cm} m: m> n/2 \text{ and } (i+2,k) \rightsquigarrow(m,k+n) \},
$$

\begin{equation*}
\Gamma_{ n}(i,k)=  \left \lbrace \begin{array}{cc}
     \text{ there exists a path connecting } (i+2,k) \rightsquigarrow \mathbb{\Z}\times\{n+k\}\text{ such that}\\
     
     \text{  this path does not intersect  the set } \{(m,s)\in \Lambda: m\leq (s-k)/2+i+1\}   \end{array}\right\rbrace
\end{equation*}
and

$$
\Gamma(i,k)=\underset{n\in \mathbb{N}}{\bigcap}\Gamma_{n}(i,k),
$$
where $(i,k)\in \Lambda$. In the rest of the paper, when $(i,k)=(0,0)$,  we omit the index in the sets $A_{n}(i,k)$, $\Gamma_n(i,k)$ and $\Gamma(i,k)$.  
 Also, we define the set
 \begin{equation*}
 C_n(N)=\left\lbrace\begin{array}{l}\text{ there exists a path connecting }[1,\infty)\times\{0\}\\ \text{ with } [1,N]\times\{n\},
 \text{ inside }([1,\infty)\times[0,\infty))\cap \Lambda \end{array}\right\rbrace.
 \end{equation*}
\begin{lemma}\label{Lemmak-dependentpercolation}
For every  $\epsilon>0$, there exists $p_0>0$ such that 
\begin{itemize}
\item[(i)]  $\hat{\pr}_p(\cup_n  A^{c}_{n})<\epsilon$ for all $p\in[p_0,1]$;
\item[(ii)] $\hat{\pr}_p(\Gamma)>1-\epsilon$, for all $p\in[p_0,1]$;
\item[(iii)] for all $p\in[p_0,1]$, there exist positive constants $c$ and $C$ such that $$\hat{\pr}_p(\{C_n(N)\}^{c})\leq Ce^{-cN},$$ for all $n$ and $N$.
\end{itemize}
\end{lemma}
\begin{proof}
Observe that in the set $\{|C_{(2,0)}|=\infty\}$ we have the following equality
\begin{align*}
&\{ m:  (2,0)\rightsquigarrow(m,n)\}\cap[\hat{l}^{2}_{n},\hat{r}^{2}_{n}]
\\&\hspace{2cm}=\{m: \exists\hspace{0.1 cm}m' \in(-\infty,2] \text{ such that } (m',0)\rightsquigarrow(m,n)\}\cap[\hat{l}^{2}_{n},\hat{r}^{2}_{n}],
\end{align*}
from where we deduce 
\begin{equation*}
\hat{r}^{(-\infty,2]}_n=\hat{r}^{\{2\}}_{n} \text{ a.s in } \{|C_{(2,0)}|=\infty\}.
\end{equation*} 
Therefore, we have 
\begin{equation}\label{estoaqui}
\hat{\pr}_p(\cup_n  A^c_{n})\leq \hat{\pr}_{p}(C_{(2,0)} \text{ is finite})+ \hat{\pr}_{p}(\exists\hspace{0.2 cm} n \geq 1: \hat{r}^{(-\infty,2]}_n < n/2).
\end{equation}
Thus, item $(i)$ follows from \eqref{estoaqui}, \eqref{limcluster} and \eqref{percolation slope}.

To prove item $(ii)$, we first observe that by the definition of the events $A_n$ we have
$$
\Gamma^{c}\subset \underset{n\geq0}{\bigcup}A^c_{n}.
$$
By item $(i)$, we have that there exists $ p_0$ such that for all $p\in(p_0,1]$
$$
\hat{\pr}_p(\Gamma^{c})\leq \hat{\pr}_p(\cup_n  A^c_{n})<\epsilon,
$$
which implies item $(ii)$.
To prove item $(iii)$, we observe that for the Bernoulli product measure it holds that
\begin{align}\label{survey}
\begin{split}
\hat{\pr}_p(\{C_n(N)\}^c)&=\hat{\pr}_p\left(\left\lbrace\begin{array}{l}
\text{there is no path connecting }[1,N]\times\{0\}\text{ with}\\
\text{ }[1,\infty)\times\{n\}\text{ inside }([1,\infty)\times [0,\infty))\cap \Lambda
\end{array}\right\rbrace\right)\\
&\leq \hat{\pr}_p\left(\left\lbrace\begin{array}{l}
\text{there is no infinite path beginning in}\\
 \text{ }[1,N]\times \{0\},\text{ inside }([1,\infty)\times [0,\infty))\cap \Lambda
\end{array}\right\rbrace\right).
\end{split}
\end{align}
Using the contour method, we obtain that, for $p$ close enough to $1$, there exist positive constants $c$ and $C$ depending on $p$ such that the last probability in \eqref{survey} is smaller than $Ce^{-cN}$ (see \cite{Survey}, page $1026$).
\end{proof}
For stating the next lemma  we  need to define the set
\begin{equation}\label{unGammadeperco}
\tilde{\Gamma}_n(i)=\left\lbrace \begin{array}{c}
 \text{there exists  a path connecting } [n/2+i,\infty)\times\{0\}\cap \Lambda \rightsquigarrow (\imath,n)
\text{ and}\\
\text{this path does not intersect the set } \{(m,s)\in \Lambda: m\leq n/2-s/2+i \}
\end{array}\right\rbrace,
\end{equation}where $\imath=\imath(i,n)=i+2$, if $i+n$ is even and $\imath=\imath(i,n)=i+1$ otherwise. When $i=0$, we omit the index in the set $\tilde{\Gamma}_n$.
\begin{lemma}\label{Lemak-pendentpercolation2}
For $\epsilon$ and $k\in \mathbb{N}$, there exists $\delta>0$ such that if $(X, \mathcal{X}, \hat{\pr})$ is a $k$-dependent oriented percolation system with closure below $\delta$, then for all positive integer $n$ we have 
\begin{itemize}
\item[(i)] $
\hat{\pr}(\tilde{\Gamma}_n)> 1-\epsilon;
$ 
\item[(ii)]$\hat{\pr}(\{C_n(N)\}^c)\leq Ce^{-cN}$ for all $N$.
\end{itemize}

\end{lemma}
\begin{proof}
To prove both items we take $p_0$ as in item $(ii)$ of Lemma \ref{Lemmak-dependentpercolation} and $\delta$ as in Lemma \ref{LiggettSmallClosure},  such that for all $k$-dependent oriented percolation system  $(X, \mathcal{X}, \hat{\pr})$ with closure under $\delta$,  $\hat{\pr}$ stochastically  dominates $\hat{\pr}_{p_0}$. 
 
To prove item $(i)$, we first suppose that $n$ is an odd positive integer. Note that for each path located to the right of the line $x=n/2-y/2$ that connects  $[ n/2,\infty)\times\{0\}$ with $(1,n)$,  we can  construct another path, to the right of the line $x=y/2+1$, connecting $(2,0)$ with $[ n/2,\infty)\times\{n\}$, see Figure \ref{elcaminoverde}. By construction, both paths have the same probability to occur under the Bernoulli product measure $\hat{\pr}_{p_0}$. Therefore
\begin{figure}[ht!]
\begin{center}
\begin{minipage}[c]{5cm}
  \begin{overpic}[scale=0.5,unit=1mm]{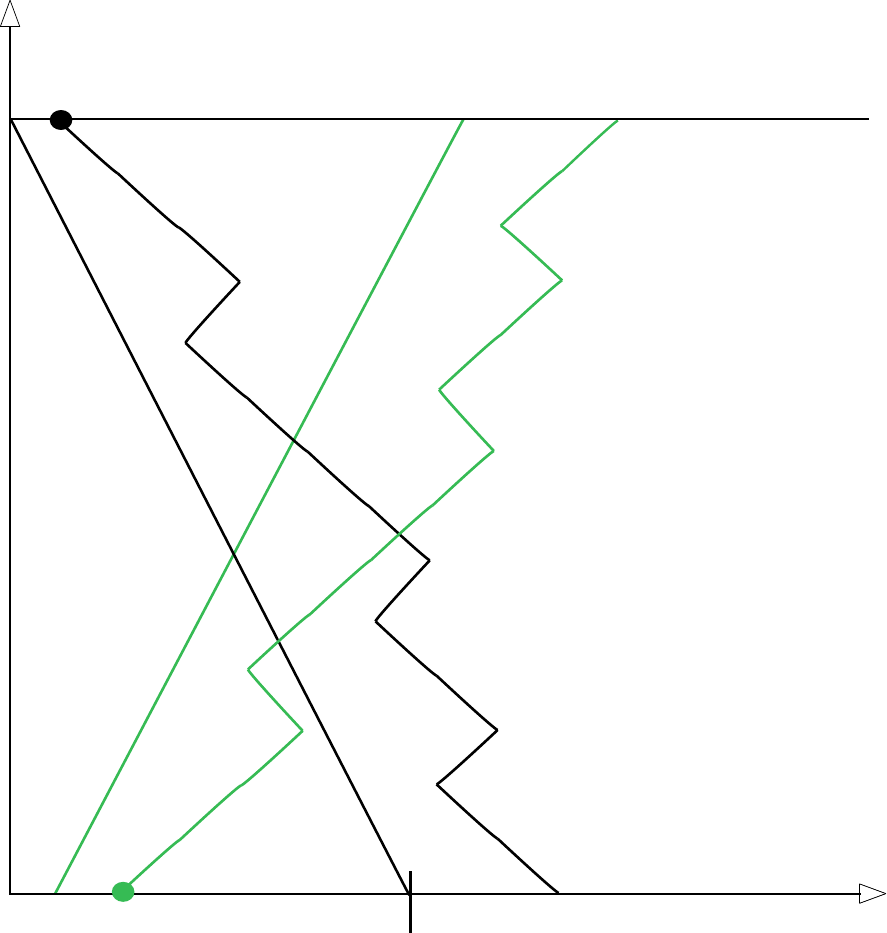}
   \put(-1.5,40.5){\parbox{0.4\linewidth}{
\begin{tiny}
$n$
\end{tiny}}}
\put(1,42){\parbox{0.4\linewidth}{
\begin{tiny}
$(1, n)$
\end{tiny}}}
\put(18,41.5){\parbox{0.4\linewidth}{
\begin{tiny}
\textcolor{miverde}{$x=\frac{y}{2}+1$}
\end{tiny}}}
\put(3.5,-0.5){\parbox{0.4\linewidth}{
\begin{tiny}
$(2,0)$
\end{tiny}}}
\put(21,-2){\parbox{0.4\linewidth}{
\begin{tiny}
$x=\frac{n}{2}-\frac{y}{2}$
\end{tiny}}}
\end{overpic}
 \end{minipage}
 \hspace{1cm}
 \begin{minipage}[c]{5cm}
 \begin{overpic}[scale=0.5,unit=1mm]{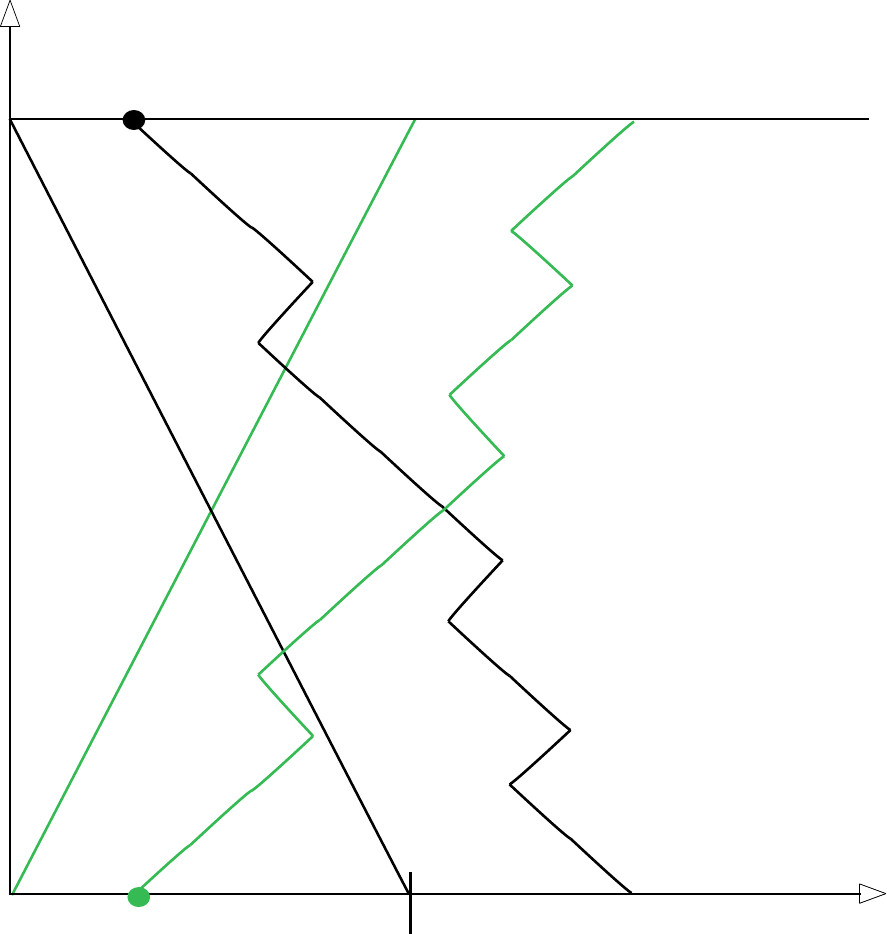}
   \put(-1.5,40.5){\parbox{0.4\linewidth}{
\begin{tiny}
$n$
\end{tiny}}}
\put(4.5,42){\parbox{0.4\linewidth}{
\begin{tiny}
$(2, n)$
\end{tiny}}}
\put(18,41.5){\parbox{0.4\linewidth}{
\begin{tiny}
\textcolor{miverde}{$x=\frac{y}{2}$}
\end{tiny}}}
\put(3.5,-0.5){\parbox{0.4\linewidth}{
\begin{tiny}
$(2,0)$
\end{tiny}}}
\put(21,-2){\parbox{0.4\linewidth}{
\begin{tiny}
$x=\frac{n}{2}-\frac{y}{2}$
\end{tiny}}}
\end{overpic}
 \end{minipage}
 \end{center}
 \caption{In the left figure, $n$ is considered odd and in the right figure, $n$ is even. In both cases, under the Bernoulli product measure, the green path at the right of the green line has the same probability that the path at the right of the black line.}

\label{elcaminoverde}
\end{figure}
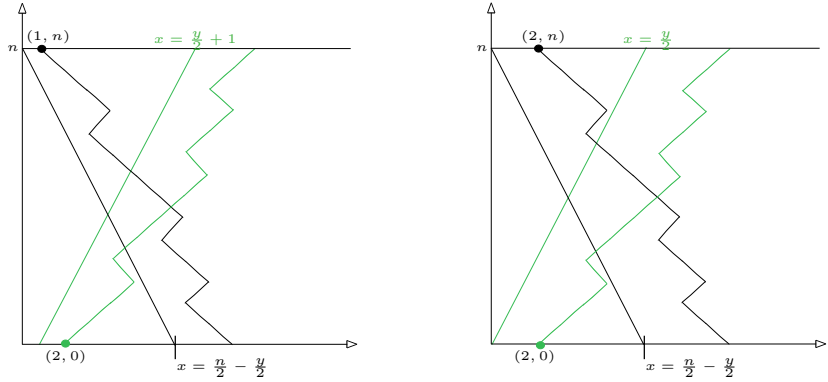
\begin{equation}\label{eq1}
\hat{\pr}_{p_0}(\tilde{\Gamma}_n)= \hat{\pr}_{p_0}(\Gamma_{n})\geq \hat{\pr}_{p_0}(\Gamma)\geq 1-\epsilon,
\end{equation}
where the last inequality in \eqref{eq1} follows by item $(ii)$ of Lemma \ref{Lemmak-dependentpercolation}. Moreover, we have that $\tilde{\Gamma}_n$ is an increasing set and therefore
\begin{align}\label{veamos}
\hat{\pr}(\tilde{\Gamma}_n)\geq \hat{\pr}_{p_0}(\tilde{\Gamma}_n).
\end{align}
Hence, \eqref{eq1} and \eqref{veamos} imply the desired lower bound.

To conclude the proof of item $(i)$,  we consider the case where $n$ is a positive even integer. Observe that for each path that is to the right of the line $x=\frac{n}{2}-\frac{y}{2}$ and connects  $[n/2,\infty)\times\{0\}$ with $(2,n)$, we can  construct another path to the right of the line $x=\frac{y}{2}$ that connects  $(2,0)$ with $[ n/2,\infty)\times\{n\}$. Also, if there exists a path to the right of the line $x=\frac{y}{2}+1$ connecting $(2,0)$ with $[ n/2,\infty)\times\{n\}$, then there exists a path to the right of the line $x=\frac{y}{2}$, connecting $(2,0)$ with $[ n/2,\infty)\times\{n\}$. Thus, we have
$$
\hat{\pr}_{p_0}(\tilde{\Gamma}_n)\geq \hat{\pr}_{p_0}(\Gamma_n)\geq 1-\epsilon.
$$
 The rest of the proof runs as in the case where $n$ is odd. 

Item $(ii)$ is a consequence of the fact that the event $C_n(N)$ is increasing, $\hat{\pr}$ dominates the measure $\hat{\pr}_{p_0}$, and item $(iii)$ of Lemma \ref{Lemmak-dependentpercolation}. 
\end{proof}
 	We now   present the Mountford-Sweet renormalization introduced in  \cite{MountfordSweet} for the contact process with range $R>1$, which is  a measurable map  with state space  $X$.  We denote this map by $\Psi$ and observe that its definition depends on two positive integers $\hat{N}$ and $\hat{K}$. 
 	
 	Let $\hat{N}$ and $\hat{K}$ be two positive integers. 
		Given $m\in \mathbb{Z}$ and $n \in \mathbb{Z}^{+}$, we define the following sets 
\begin{align*}
&\mathcal{I}^{\hat{N}}_m=\left(\frac{m\hat{N}}{2}-\frac{\hat{N}}{2},\frac{m\hat{N}}{2}+\frac{\hat{N}}{2}\right]\cap \mathbb{Z},\quad
I^{\hat{N},\hat{K}}_{(m,n)}=\mathcal{I}^{\hat{N}}_m\times \{\hat{K}\hat{N}n\},\\&
J^{\hat{N},\hat{K}}_{(m,n)}=\left(\frac{m\hat{N}}{2}-R, \frac{m\hat{N}}{2}+R\right)\times [\hat{K}\hat{N}n, \hat{K}\hat{N}(n+1)].
\end{align*}
We call the set 
$$
I^{\hat{N},\hat{K}}_{(m,n)}\cup J^{\hat{N},\hat{K}}_{(m,n)}\cup I^{\hat{N},\hat{K}}_{(m,n+1)}
$$
the renormalized box corresponding to $(m,n)$, or just the box $(m,n)$.

To define $\Psi$ we start considering an auxiliary map $\Phi\in \{0,1,2\}^{\Lambda}$. Given $(m,0)\in \Lambda$, set $\Phi(m,0)=1$ if the following conditions are satisfied:
\begin{align}
&\text{for each interval }I\subset \mathcal{I}^{\hat{N}}_{m-1} \cup \mathcal{I}^{\hat{N}}_{m+1} \text{ of length } \sqrt{\hat{N}} \text{ it holds } I \cap \xi^{\mathbb{Z}}_{\hat{N}\hat{K}}\neq \emptyset;
\label{MSa0}\\
&\begin{array}{c}
\text{if }x\in  \mathcal{I}^{\hat{N}}_{m-1} \cup \mathcal{I}^{\hat{N}}_{m+1}\text{ and }\mathbb{Z}\times{\{0\}}\rightarrow (x,\hat{K}\hat{N}),
\text{ then } I^{\hat{N},\hat{K}}_{(m,0)}\rightarrow (x,\hat{K}\hat{N});
\end{array}\label{MSb0}\\
&\text{if }(x,s) \in J^{\hat{N},\hat{K}}_{(m,0)} \text{ and }\mathbb{Z}\times{\{0\}}\rightarrow (x,s)
\text{ then }I^{\hat{N},\hat{K}}_{(m,0)}\rightarrow (x,s);
\label{MSc0}\\
&\left\lbrace\begin{array}{c}
x\in \mathbb{Z}:\exists s, t,0\leq s<t\leq \hat{K}\hat{N},\\
 y\in  \mathcal{I}^{\hat{N}}_{m-1} \cup \mathcal{I}^{\hat{N}}_{m+1}  \text{ such that } (x,s)\rightarrow (y,t)
\end{array}\right\rbrace\subset\left[ \frac{m\hat{N}}{2}-2\alpha \hat{K}\hat{N}, \frac{m\hat{N}}{2}+2\alpha \hat{K}\hat{N} \right].\label{MSd0}
\end{align}
 Otherwise, set $\Phi(m,0)=0$. Given $(m,n)\in \Lambda$ with $n\geq 1$, set $\Phi(m,n)=1$ if
\begin{align}
&1\in \{\Phi(m-1,n-1),\Phi(m+1,n-1)\};\label{Msa'}\\
&\text{for each interval }I\subset \mathcal{I}^{\hat{N}}_{m-1} \cup \mathcal{I}^{\hat{N}}_{m+1} \text{ of length } \sqrt{\hat{N}} \text{ it holds } I \cap \xi^{\mathbb{Z}}_{\hat{K}\hat{N}(n+1)}\neq \emptyset;
\label{MSa}\\
&\begin{array}{c}
\text{if }x\in  \mathcal{I}^{\hat{N}}_{m-1} \cup \mathcal{I}^{\hat{N}}_{m+1}\text{ and }\xi^{\mathbb{Z}}_{\hat{K}\hat{N}n}\times{\{\hat{K}\hat{N}n\}}\rightarrow (x,\hat{K}\hat{N}(n+1)),\\
\text{ then }(\xi^{\mathbb{Z}}_{\hat{K}\hat{N}n}\times{\{\hat{K}\hat{N}n\}})\cap I^{\hat{N},\hat{K}}_{(m,n)}\rightarrow (x,\hat{K}\hat{N}(n+1));
\end{array}\label{MSb}\\
&\text{if }(x,s) \in J^{\hat{N},\hat{K}}_{(m,n)} \text{ and }\xi^{\mathbb{Z}}_{\hat{K}\hat{N}n}\times{\{\hat{K}\hat{N}n\}}\rightarrow (x,s)
\text{ then }(\xi^{\mathbb{Z}}_{\hat{K}\hat{N}n}\times{\{\hat{K}\hat{N}n\}})\cap I^{\hat{N},\hat{K}}_{(m,n)}\rightarrow (x,s);
\label{MSc}\\
    &\left\lbrace\begin{array}{c}
x\in \mathbb{Z}:\exists s, t,\hat{K}\hat{N}n\leq s<t\leq \hat{K}\hat{N}(n+1),\\
 y\in  \mathcal{I}^{\hat{N}}_{m-1} \cup \mathcal{I}^{\hat{N}}_{m+1}  \text{ such that } (x,s)\rightarrow (y,t)
\end{array}\right\rbrace\subset\left[ \frac{m\hat{N}}{2}-2\alpha \hat{K}\hat{N}, \frac{m\hat{N}}{2}+2\alpha \hat{K}\hat{N} \right].\label{MSd}
\end{align}
If \eqref{Msa'} fails set $\Phi(m,n)=2$, and in every other case set $\Phi(m,n)=0$. Finally, we define 

\begin{equation*}
\Psi(m,n)=\left\lbrace
\begin{array}{cl}
0, & \text{ if }\Phi(m,n)=0\\
1, & \text{ otherwise}.
\end{array}\right.
\end{equation*}
We now make several remarks about the conditions in the definition of $\Psi$. First, equations \eqref{MSa0} and \eqref{MSa} imply that there are many sites at the top of the boxes  $(m-1,n)$ and $(m+1,n)$ that are connected in the Harris construction with $\mathbb{Z}\times\{ 0\}$. Second, equations \eqref{MSb0} and \eqref{MSb} yield that if  a site at the top of the box  $(m,n)$ is connected in the Harris construction with $\mathbb{Z}\times \{0\}$, then it is connected with the base of the box $(m,n)$. Third, equations \eqref{MSc0} and  \eqref{MSc} guarantee that if a  site in the rectangle $J^{\hat{N},\hat{K}}_{(m,n)}$ is connected with $\mathbb{Z}\times \{0\}$, then it is connected with the base of the  box $(m,n)$. Finally, equations \eqref{MSd0} and \eqref{MSd} imply that every path with initial time larger than $\hat{K}\hat{N}n$ and final point in the  box  $(m,n)$ is inside the rectangle 
\begin{equation}\label{rect}\left[ \frac{m\hat{N}}{2}-2\alpha \hat{K}\hat{N}, \frac{m\hat{N}}{2}+2\alpha \hat{K}\hat{N} \right]\times[\hat{K}\hat{N}n,\hat{K}\hat{N}(n+1)].
\end{equation}
 The  rectangle in \eqref{rect} is called the envelope of the box $(m,n)$. Additionally, we observe  that the constant  $\alpha$ in equation \eqref{MSd} is as in \eqref{defalpha}. 
 
\begin{proposition}[See \cite{MountfordSweet}]
\label{Mountford-Sweet}
There exist $k$ and $\hat{K}$ with the property that, for any $\delta>0$ there is $\hat{N}_0$ such that the law of $\Psi$ is a $k$-dependent percolation system with closure under $\delta$ for all $\hat{N}>\hat{N}_0$. 
\end{proposition}
We note that for $\epsilon>0$ if we choose
\begin{align}
&k \text{ and } \hat{K} \text{ as in Proposition }\ref{Mountford-Sweet};\label{hat{K}}\\
&p_0\text{ as in Lemma }\ref{Lemmak-dependentpercolation};\label{p_0}\\
&\delta=\delta(k,p_0,\epsilon)\text{ as in Lemma }\ref{LiggettSmallClosure} \text{  and Lemma }\ref{Lemak-pendentpercolation2};\label{delta}\\
&\hat{N}_0=\hat{N}_0(\delta,k,\hat{K})\text{ as in Proposition }\ref{Mountford-Sweet}\label{hat{N}};\\
&\hat{N}>\hat{N}_0\label{hatN};
\end{align}
 we have that the law of  $\Psi$ is a $k$\emph{-dependent} percolation system with closure under $\delta$ and it is stochastically larger than $\hat{\pr}_{p_0}$. Also, the law of $\Psi$ satisfies the statement of  Lemma \ref{Lemak-pendentpercolation2}.
 
 Next, we  recall the definition of expanding point that appears in \cite{Conos}. Before this definition, we introduce the following sets in oriented percolation
\begin{equation*}
\Gamma^{-}_{ n}(i,k)=  \left \lbrace \begin{array}{cc}
     \text{ there exists a path connecting } (i-2,k) \rightsquigarrow \mathbb{\Z}\times\{n+k\}\text{ such that}\\
     
     \text{  this path does not intersect  the set } \{(m,s)\in \Lambda: m\geq -(s-k)/2+i-1\}   \end{array}\right\rbrace
\end{equation*}
and
$$
{\Gamma}^-(i,k)=\underset{n\in \mathbb{N}}{\bigcap}\Gamma^-_{n}(i,k),
$$
where $(i,k)$ is a point in $\Lambda$. 
\begin{definition}\label{Def:expanding}
The point $(x,t)\in \mathbb{Z}\times \mathbb{R}^{+}$  is \emph{expanding} if:
\begin{enumerate}
\item[(1)] for all $z \in  \mathcal{I}^{\hat{N}}_{i-2} \cup \mathcal{I}^{\hat{N}}_{i} \cup \mathcal{I}^{\hat{N}}_{i+2}$,  $(x, t) \leftrightarrow (z, k\hat{K}\hat{N})$ inside $\mathcal{I}^{\hat{N}}_{i-2} \cup \mathcal{I}^{\hat{N}}_{i} \cup \mathcal{I}^{\hat{N}}_{i+2}$, where  $k=k(t)=\left\lceil \frac{t}{\hat{K}\hat{N}}\right\rceil$ and $i=i(x,t)$ is such that $(i,k)\in \Lambda$ and $x\in\mathcal{I}^{\hat{N}}_{i}$;
\item[(2)] $\Psi \in \Gamma(i,k)\cap \Gamma^-(i,k)$.
\end{enumerate}
\end{definition}If $(x,t)$ is expanding, we call the cone $$\lbrace(y,s) \in \mathbb{Z}\times[t,\infty) :\, x - s/2 \leq  y \leq x + s/2\rbrace,$$ 
 the \emph{descendency barrier} of $(x,t)$.

Furthermore, we call the point $(x,t)$ \emph{expanding to the right}  if the property $(1)$ in Definition \ref{Def:expanding} is satisfied and  $\Psi\in \Gamma(i,k)$. Similarly, we call the point $(x,t)$ \emph{expanding to the left} if the first property in Definition \ref{Def:expanding} is satisfied and  $\Psi \in \Gamma^-(i,k)$. 

The proof of the following proposition can be found  in \cite{MountfordSweet} and \cite{Conos}.
\begin{proposition}\label{AMPV}
For any $R\in \mathbb{N}$ and $\lambda >\lambda_ c (\mathbb{Z}, R)$, there exists $\overline{\delta}> 0 $  such that
\begin{equation*}
\pr((0,0) \text{ is }\text{\emph{expanding}})>\overline{\delta}.
\end{equation*}
Also, for every $\epsilon>0$ there exists $N$ such that, for all $A$ subset of $\mathbb{Z}$ with $|A|\geq N$ 
$$\pr(\text{no point of }A\times\{0\}\text{ is \emph{expanding}})<\epsilon.$$
\end{proposition}
\begin{corollary}\label{CorollaryAMPV}
For any $R\in \mathbb{N}$ and $\lambda >\lambda_ c (\mathbb{Z}, R)$, there exists $\overline{\delta}> 0 $  such that
\begin{equation*}
\pr((0,0) \text{ is }\text{\emph{expanding to the right}})>\overline{\delta}.
\end{equation*}
Also, for every $\epsilon>0$ there exists $N\geq 1$ such that, for all $A$ subset of $\mathbb{Z}$ with $|A|\geq N$ 
\begin{equation}\label{elncito}
\pr(\text{no point of }A\times\{0\}\text{ is  \emph{expanding to the right}})<\epsilon.
\end{equation}
\end{corollary}

\section{ Convergence results}\label{Convergence results}
This section has two subsections. In Subsection \ref{Proof of invariant measure}, we establish Proposition \ref{tightness}, which is the key result to prove the existence of  an invariant measure $\nu$ with infinitely many particles of type $1$ and $2$. In Subsection \ref{Proof of main Theorem}, we present the technical Proposition \ref{lallave} that is essential to obtain our main results. In this subsection we also prove Theorem \ref{teoconvergence}.
\subsection{Existence of the invariant measure $\nu$}\label{Proof of invariant measure}
To simplify notation, we denote by $\textbf{r}^1_{t}$  the rightmost site at time $t$ occupied by a particle of type $1$ for the two-type contact process with initial configuration $\mathds{1}_{(-\infty,0]}+2\mathds{1}_{[1,\infty)}$. We denote by $\textbf{l}^2_{t}$ the leftmost site at time $t$ occupied by a particle of type $2$ for the two-type contact process with initial configuration $\mathds{1}_{(-\infty,0]}+2\mathds{1}_{[1,\infty)}$. By the symmetry of the Harris construction, this variable has the same distribution as $-\textbf{r}^1_{t}+1$. Therefore, the  next proposition is also valid for $-\textbf{l}^2_{t}$.
\begin{proposition}\label{tightness}
There exists $M$ such that
\begin{equation*}
\pr(\textbf{r}^{1}_t\geq M N)\leq Ce^{-cN},
\end{equation*}
for $t\geq \hat{N}\hat{K}$ and for all $N$.
\end{proposition}
\begin{proof}
For $\epsilon>0$ we choose $k$, $\hat{K}$, $p_0$, $\delta$ and $\hat{N}$ as in \eqref{hat{K}}, \eqref{p_0}, \eqref{delta}, and  \eqref{hatN}, such that the law of $\Psi$ satisfies item $(ii)$ of  Lemma \ref{Lemak-pendentpercolation2}. Then, by the  translation invariance of the Mountford-Sweet renormalization, we have that
 \begin{align}\label{eqimp}
 \begin{split}
\pr\left(\Psi\in \left\lbrace\begin{array}{l}\text{ there exists a path connecting }[i,\infty)\times\{0\}\\ \text{ with } [i,N]\times\{k\},
 \text{ inside }([i,\infty)\times[0,\infty))\cap \Lambda \end{array}\right\rbrace^{c}\right)=\hat{\pr}(\{C_{k}(N-i)\}^{c})&\leq Ce^{-c(N-i)}
 \\&=\hat{C}e^{-cN},
 \end{split}
 \end{align}
for all $k$, where $i= \left \lceil 2 \alpha \hat{K} \hat{N}  \right\rceil+\left \lceil 2 \alpha \hat{K} \hat{N}  \right\rceil\mod 2$.

We take $t\geq \hat{K}\hat{N}$ and $k=k(t)=\lfloor t/\hat{K}\hat{N}\rfloor+1$. By \eqref{eqimp}  we have that, except for an event with probability smaller than $\hat{C}e^{-cN}$, there exists a sequence $\{m_j\}_{0\leq j\leq k}$ such that
\begin{align*}
&\Psi(m_j,j)=1 \hspace{0.2cm} \forall \hspace{0.2cm} j\hspace{0.2cm}\in \{0,\dots,k\},\\&
||m_{j+1}-m_{j}||=1,\quad
m_0\geq i, \quad i\leq m_k\leq N,\quad \text{ and }\quad m_j\geq i,\,  \forall \hspace{0.2cm} j\hspace{0.2cm}\in \{0,\dots,k\}.
\end{align*}
 We define the union of the renormalized boxes as 
 $$
 R_{k}=\underset{0\leq j \leq k}{\bigcup}\left(I^{\hat{K},\hat{N}}_{(m_j,j)}\cup J^{\hat{K},\hat{N}}_{(m_j,j)}\cup I^{\hat{K},\hat{N}}_{(m_j,j+1)}\right).
 $$
The set $R_k$ is connected and all the boxes have width larger than $R$. Hence, if a path begins to the left of $R_k$, ends in a point to the right of $R_k$ and has time coordinate smaller than $t$, then this path intersects $R_k$. Also, properties \eqref{MSb} and \eqref{MSc} of the Mountford-Sweet renormalization imply that in the trajectory of the contact process $t\mapsto \xi(t)(\mathcal{H})$, every point in $R_{k}$ that is connected   with $\mathbb{Z}\times \{0\}$  is connected with $I^{\hat{K},\hat{N}}_{(m_0,0)}$. Also, property \eqref{MSd} of the renormalization, our choice of $i$, and the fact that $m_j\geq i$, for all  $j$, imply that such points are connected with $I^{\hat{K},\hat{N}}_{(m_0,0)}$ by paths that are inside of $\{i\hat{N}/2-2 \alpha \hat{K}\hat{N}\} \times[0,\infty)\subset \mathbb{Z}^+ \times [0,\infty)$.
 \begin{figure}
			\begin{center}
		\begin{overpic}[scale=0.4,unit=1mm]{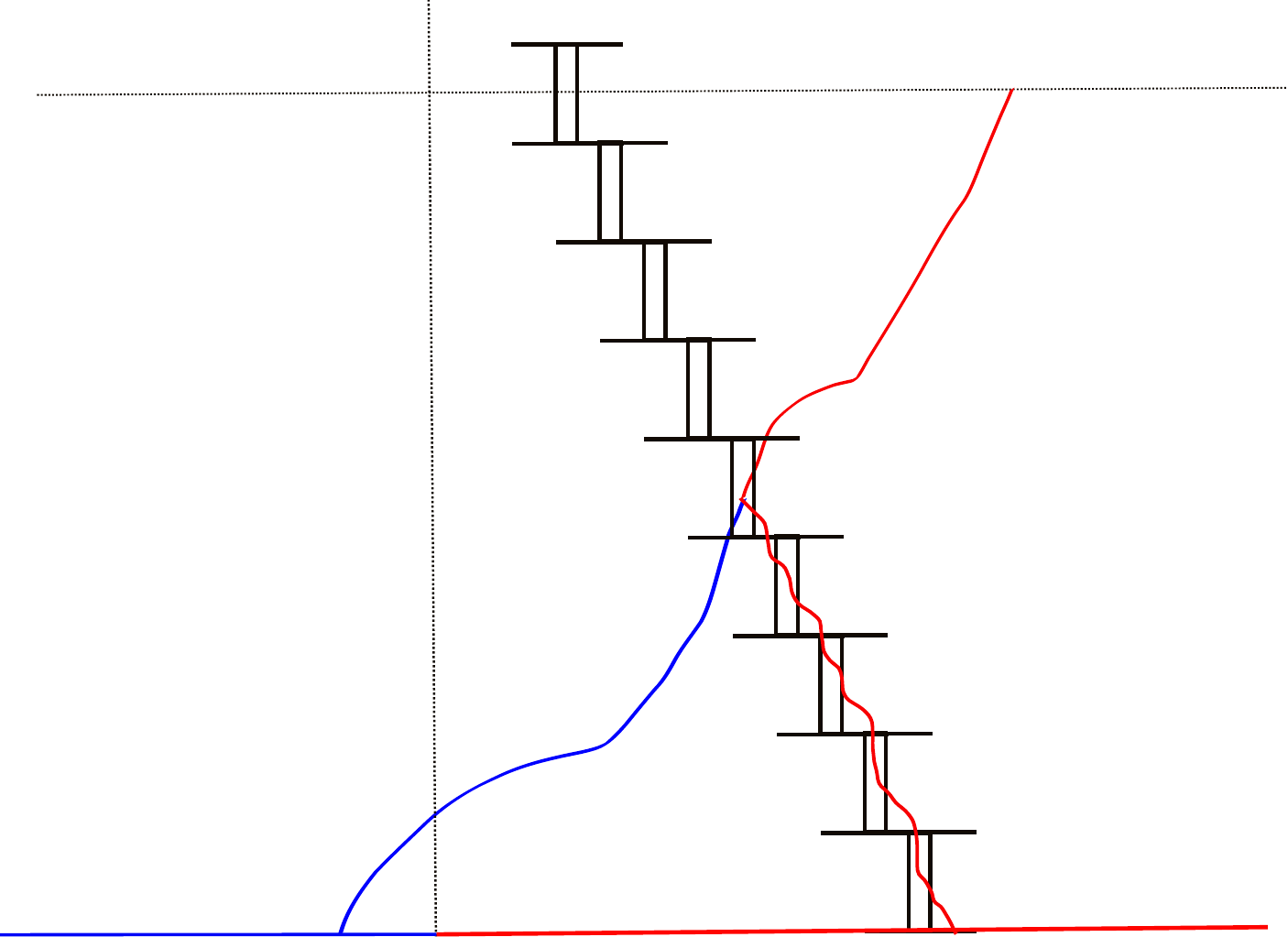}
		\put(28,38){\parbox{0.4\linewidth}{%
 \begin{tiny}
{$R_k$}
\end{tiny}}}
  \end{overpic}
  \end{center}
  \caption{The blue paths represent the particles of type $1$ and the red ones represent the particles of type $2$.}
		\end{figure}
		
 Observe that for any $y\in \xi^{\mathbb{Z}}_t\cap [N\hat{N}/2,\infty)$ and any path that connects  $\mathbb{Z}\times \{0\}$ with $(y,t)$ we have two possibilities: the path intersects $R_k$ or it stays forever to the right of $R_k$. In both cases, we can construct a path contained in $[i\hat{N}/2-2 \alpha \hat{K}\hat{N},\infty)\times[0,\infty)$ and consequently $\zeta^{\textbf{1},\textbf{2}}_t(y)=2$. Therefore, if a site in $ [N\hat{N}/2,\infty)$ is occupied at time $t$, then it is occupied by a particle of type $2$ and we conclude 
\begin{equation*}
\pr\left( \textbf{r}^1_t \geq N\hat{N}/2\right)\leq \hat{\pr}(\{C_k(N)\}^c) \leq \hat{C}e^{-cN},
\end{equation*}
for all $t\geq \hat{K}\hat{N}$ and all $N$.
\end{proof}
\begin{remark}\label{remulti}
For the case $R=1$ there is a simpler proof of Proposition \ref{tightness}. We denote by $\Xi^{\mathbb{N}}_t$ the contact process restricted to $\mathbb{N}$ and $\mu_{\mathbb{N}}$ the no trivial invariant measure of the contact process restricted to $\mathbb{N}$. Observe that
$$
\{r^{\textbf{1}}_t>N\}\subset \{\Xi^{\mathbb{N}}_t\cap [0,N]=\emptyset\}.
$$
By attractiveness we have that
$$
\pr(\Xi^{\mathbb{N}}_t\cap [0,N]=\emptyset)\leq \mu_{\mathbb{N}}(\Xi:\Xi\cap[0,N]=\emptyset).
$$
The result follows by Lemma $2.6$ in \cite{Andjel}.
\end{remark}
We denote by $\mathcal{A}$ the set of configurations in $\{0,1,2\}^{\mathbb{Z}}$ with infinite sites occupied, and for which there exists $K$ such that all the occupied sites to the right of $K$ are occupied by particles of type $2$, and all the occupied sites to the left of $-K$ are occupied by particles of type $1$. More precisely
\begin{equation}\label{defmathcalA} 
\mathcal{A}=\left\lbrace \begin{array}{l} \zeta \in \{0,1,2\}^{\mathbb{Z}}:||\textbf{l}^{2}(\zeta)||<\infty,\,||\textbf{r}^{1}(\zeta)||<\infty,\,|\{x:\zeta(x)\neq 0\}\cap[1,\infty)|=\infty,\\
\hspace{2,1cm}\text{ and }\,|\{x:\zeta(x)\neq 0\}\cap(-\infty,0]|=\infty\end{array}\right\rbrace.
\end{equation}
The following corollary is a consequence of  Proposition \ref{tightness}.
\begin{corollary}
There exists an invariant measure $\nu$ for the two-type contact process supported in the set of configurations $\mathcal{A}$.
\end{corollary}
 \begin{proof}
 We consider the metric  space $(\{0,1,2\}^{\mathbb{Z}},\tilde{\rho})$, where the distance $\tilde{\rho}$ is defined by
 $$
 \tilde{\rho}(\zeta, \zeta')= \underset{x\in \mathbb{Z}}{\sum}\frac{||\zeta(x)-\zeta'(x)||}{2^{||x||}(1+||\zeta(x)-\zeta'(x)||)},
 $$
 for every $\zeta, \zeta' \in \{0,1,2\}^{\mathbb{Z}}$. In this space,  for any Cauchy sequence, the pointwise limit is also the limit in the metric $\tilde{\rho}$.  Hence, this metric space is complete. Moreover, by the definition of the metric, $\tilde{\rho}(\zeta,\zeta ')\leq 2$ for all $\zeta, \zeta'$. Therefore, $(\{0,1,2\}^{\mathbb{Z}},\tilde{\rho})$ is a compact metric space. 
 
 We denote by $\nu_t$ the law of $\zeta^{\textbf{1},\textbf{2}}_t$ and for $T\geq 0$ we define the measure
 $$
 \tilde{\nu}_T(A)= \frac{1}{T} \int^{T}_{0}\nu_t(A)dt,
 $$
for any Borel set $A$. Since the space is compact, $\{\tilde{\nu}_T\}_T$ is a tight family of probabilities. Let $\{\tilde{\nu}_{T_k}\}_k$ be a subsequence that converges  to a measure $\nu$. Using Proposition $1.8$ of Chapter $I$ in \cite{Liggett},  we have that $\nu$ is an invariant measure for the process.
 
 It remains to prove that the measure $\nu$ is supported in $\mathcal{A}$.
 For $\epsilon>0$, we take $MN$ as in Proposition \ref{tightness} such that for all $t\geq 0$ we have 
 \begin{equation}\label{eq13}
 \pr(\textbf{r}^1_t> MN) +\pr(\textbf{l}^2_{t}<-MN)\leq \frac{\epsilon}{2}.
\end{equation}
 Observe that 
  \begin{equation}\label{mamayoquiero}
  \left\lbrace\begin{array}{l}|\xi^{\mathbb{Z}}_t\cap(-\infty,-MN]|=\infty;|\xi^{\mathbb{Z}}_t\cap[MN,\infty)|=\infty;\\
  \textbf{r}^1_t<MN;\,\textbf{l}^2_t>-MN\end{array}\right\rbrace\subset \{\zeta^{\textbf{1},\textbf{2}}_t\in \mathcal{A}\}.
\end{equation}
 The event that there is no mark before time $t$ for a Poisson process of rate $1$ has probability $e^{-t}$. Since all the Poisson processes of death are independent, the smallest site $x\in[MN,\infty)$ for which there is no mark of death before time $t$ has geometric distribution with parameter of success  $e^{-t}$. Similarly, we have that the $n$-th smallest site in $[MN,\infty)$ for which there is no mark of death before time $t$ has negative binomial distribution with parameters $n$ and $e^{-t}$.  Consequently, with probability $1$ and for any $n$, there are at least $n$ sites in  $[MN, \infty)$ occupied at time $t$ by the process $\{\xi^{\mathbb{Z}}_t\}$. Further, since $n$ is arbitrary, with probability $1$, there are infinite sites in  $[MN, \infty)$ occupied at time $t$ by the process $\{\xi^{\mathbb{Z}}_t\}$. By the symmetry of the Harris graph, this argument is also valid for $(-\infty,-MN]$, and we conclude that
$$
\pr(|\xi^{\mathbb{Z}}_t\cap(-\infty,-MN]|=\infty;|\xi^{\mathbb{Z}}_t\cap[MN,\infty)|=\infty)=1.
$$
The equation above, \eqref{mamayoquiero} and \eqref{eq13} imply that
$$
\pr(\zeta^{\textbf{1},\textbf{2}}_t\in \mathcal{A})\geq 1-\epsilon,
$$
for all $t$ and, therefore, we have 
  \begin{equation}\label{majela}
  1-\epsilon\leq \tilde{\nu}_{T_k}(\mathcal{A}).
  \end{equation}
Since $\mathcal{A}$ is a closed set in $\{0,1,2\}^{\mathbb{Z}}$ and $\tilde{\nu}_{T_k}$ converges to $\nu$, \eqref{majela} implies that $\nu(\mathcal{A})\geq 1-\epsilon$ for $\epsilon$ arbitrary, which completes the proof. 
 \end{proof}

\subsection{Proof of Theorem \ref{teoconvergence}}\label{Proof of main Theorem}
Before the proof of Theorem \ref{teoconvergence}, we need to state several technical results. The most important is Proposition \ref{lallave}, which will be essential to obtain Theorem \ref{teoconvergence} and Theorem \ref{Teo2}.

We begin by introducing some notations. Let $M'$ be a positive number, we define for $k\geq 1$
\begin{equation}\label{A_k}
B_k=\left\lbrace \begin{array}{l}\omega: P^x(\omega)\cap[(k-1)\hat{K}\hat{N},k\hat{K}\hat{N}]\neq \emptyset;
P^{y\rightarrow x}(\omega)\cap[(k-1) \hat{K}\hat{N},k\hat{K}\hat{N}]=\emptyset;\\
\hspace{0.5cm}\forall \, x\in [-M',M'];|| x-y ||\leq R\end{array}\right\rbrace.
\end{equation}
In the event $B_k$, for every site in $[-M',M']$, there is at least one mark of death in the time interval $[(k-1) \hat{K}\hat{N},k\hat{K}\hat{N}]$, and there are no arrows coming from a site outside $[-M',M']$ to a site in $[-M',M']$ during this time interval. Therefore, in the event $B_k$,  there is no  point in  $[-M',M']\times\{k\hat{K}\hat{N}\}$ connected with $\mathbb{Z}\times\{0\}$ in the Harris graph.

During this section we take $\epsilon$ a positive arbitrary number. For  $\epsilon$ we take $\hat{K}$ and $\hat{N}$ as in \eqref{hat{K}} and \eqref{hatN}, respectively, and we define
\begin{align}\label{elv}
v=\max\{3,4\alpha\hat{K}\}\quad\text{ and }\quad M'= (v+2)\hat{N}.
\end{align}

With the quantities $v$ and $M'$ we define the stopping time
\begin{equation*}
    \textbf{X}_{k_0}(\omega)=\min\{k\hat{K}\hat{N}: \,k\geq k_0\text{ and }\omega\in B_k\},
\end{equation*}
where  $k_0\in \mathbb{N}$. The stopping time $\textbf{X}_{k_0}$ is defined such that  the rectangle $[-M',M']\times[\textbf{X}_{k_0}-\hat{K}\hat{N},\textbf{X}_{k_0}]$ has no arrows coming in or out and every site in $[-M',M']$ has a death mark. 

In the next lemma, we state a result for the Mountford-Sweet renormalization. To do this, we need to define the following set in oriented percolation 

\begin{equation*}
\mathcal{R}(\tilde{\Gamma}_k)(i)=\left\lbrace \begin{array}{c}
 \text{there exists a path connecting }(-\infty,-k/2 -i]\times\{0\}\cap \Lambda \rightsquigarrow (-\imath ,k)\text{ and}\\
\text{this path does not intersect the set }\{(m,s)\in \Lambda:-k/2+s/2-i \leq m\}
\end{array}\right\rbrace,
\end{equation*}
where $\imath=i+2$ if $i+k$ is even and $\imath=i+1$ otherwise. Reflecting in the axis of time a path in the event $\mathcal{R}(\tilde{\Gamma}_k)(i)$, we get a path in the event $\tilde{\Gamma}_k(i)$ defined in \eqref{unGammadeperco}, and vice-versa.  Since the law of the Harris graph is invariant under reflections in the axis of time, each of these paths has the same probability under the law of $\Psi$. Hence, the events $\{\Psi\in \mathcal{R}(\tilde{\Gamma}_k)(i)\}$ and $\{\Psi\in\tilde{\Gamma}_k(i)\}$ have equal probability.   
 \begin{lemma}\label{otroauxiliar} For $\epsilon>0$  we have that
 \begin{equation}\label{majorev0}
 \pr\left(\Psi\in \tilde{\Gamma}_{\textbf{x}-2}(a)\cap \mathcal{R}(\tilde{\Gamma}_{\textbf{x}-2})(a)\right)\geq1-2\epsilon,
 \end{equation}
 where $\textbf{x}=\frac{\textbf{X}_{k_0}}{\hat{K}\hat{N}}$ and $a=\left\lfloor v\right\rfloor$.
 \end{lemma}
 \begin{proof}
  Let $B_k$ be as in \eqref{A_k}. Since $X_{k_0}$ is finite almost surely we have
\begin{figure}[h]
\begin{center}
  \begin{overpic}[scale=0.4,unit=1mm]{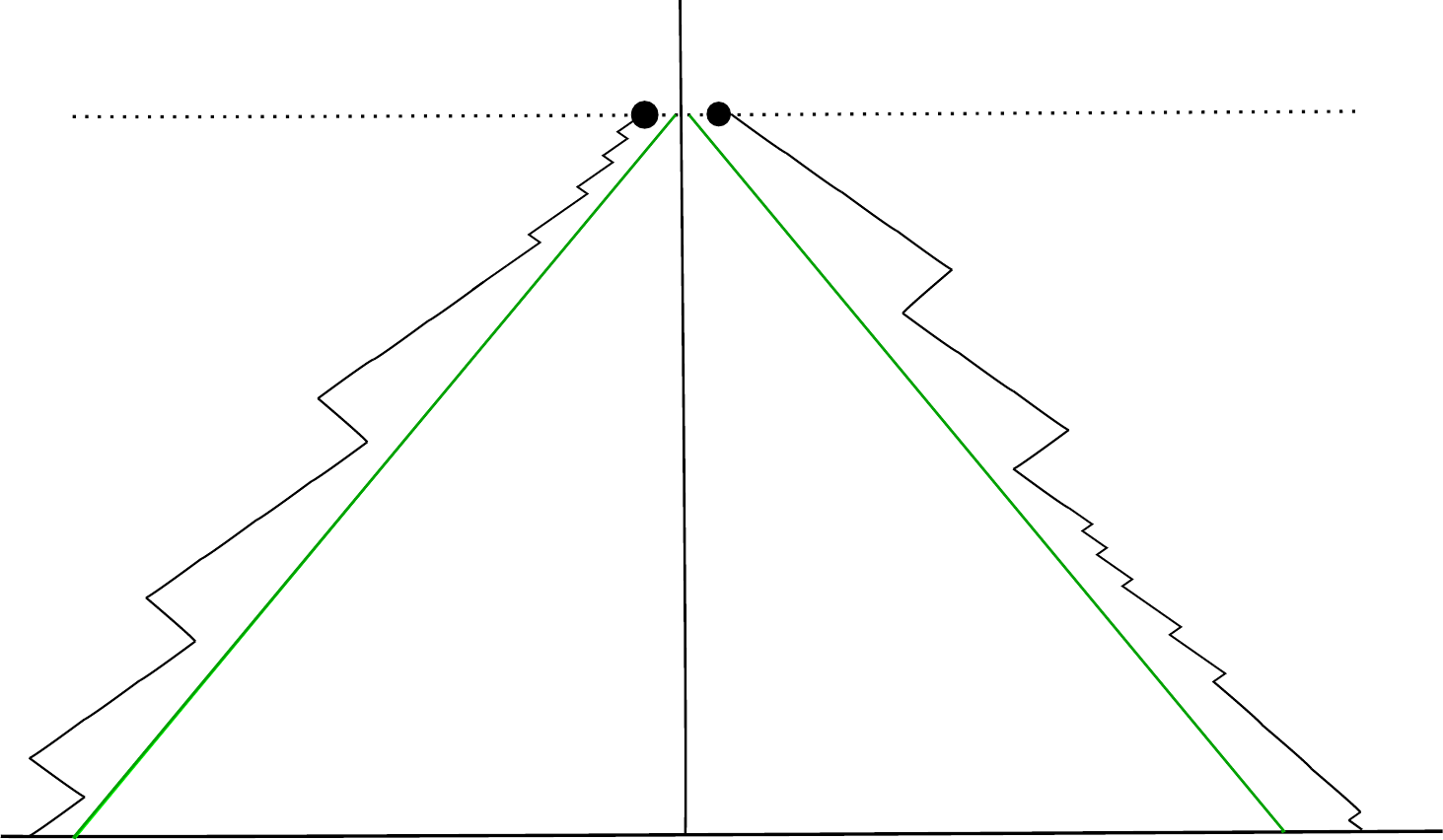}
   \put(30,31){\parbox{0.4\linewidth}{
\begin{tiny}
$\imath$
\end{tiny}}}
\put(24,31){\parbox{0.4\linewidth}{
\begin{tiny}
$-\imath$
\end{tiny}}}
\put(-2,29){\parbox{0.4\linewidth}{
\begin{tiny}
$\textbf{x}-2$
\end{tiny}}}
\end{overpic}
\caption{ Representation of the event inside the probability in \eqref{majorev0}. The green lines in the figure have equations $x=-\frac{\textbf{x}-2}{2}+ \frac{y}{2}-a$ and $x=\frac{\textbf{x}-2}{2}- \frac{y}{2}+a$, respectively. The constant $\imath$ depends on $a$ and $\textbf{x}-2$, as in the definition of the event $\mathcal{R}(\tilde{\Gamma}_{\textbf{x}-2})(a)$.}\label{Figura0}
\end{center}
\end{figure}
\begin{align*}
     &\pr(\Psi\in \tilde{\Gamma}_{\textbf{x}-2}(a)\cap \mathcal{R}(\tilde{\Gamma}_{\textbf{x}-2})(a) )=\sum\limits^{\infty}_{k=k_0} \pr(\Psi \in \tilde{\Gamma}_{k-2}(a)\cap \mathcal{R}(\tilde{\Gamma}_{k-2})(a);\textbf{X}_{k_0}=k\hat{K}\hat{N})\\
     &= \pr(\Psi\in \tilde{\Gamma}_{k_0-2}(a)\cap \mathcal{R}(\tilde{\Gamma}_{k_0-2})(a);B_{k_0})+\sum\limits^{\infty}_{k=k_0+1}\pr(\Psi\in \tilde{\Gamma}_{k-2}(a)\cap \mathcal{R}(\tilde{\Gamma}_{k-2})(a);B_k\cap^{k-1} _{j=k_0}B^c_j).
\end{align*}
The event $B_k$ is independent of any event that depends on the Harris graph until time $(k-1)\hat{K}\hat{N}$, therefore, we have 
 \begin{align}\label{otracosamasquenosequeponerle}
 \begin{split}
& \pr(\Psi\in \tilde{\Gamma}_{k-2}(a)\cap \mathcal{R}(\tilde{\Gamma}_{k-2})(a);B_k\cap^{k-1} _{j=k_0}B^c_j)=\pr(B_k)\pr(\Psi \in \tilde{\Gamma}_{k-2}(a)\cap \mathcal{R}(\tilde{\Gamma}_{k-2})(a);\cap^{k-1} _{j=k_0}B^c_j).
 \end{split}
 \end{align}
Moreover, the event  $\cap^{k-1} _{j=k_0}B^c_j$ is increasing\footnote{For the definition of an increasing event in the Harris graph  see in  page 248, \cite{olivieri2005large}.}. Also, the event in \eqref{otracosamasquenosequeponerle} that depends on $\Psi$  is increasing. These observations, \eqref{otracosamasquenosequeponerle}, and the FKG inequality imply that\begin{align}\label{algoaquiaux}
 \begin{split}
&\pr(\Psi\in \tilde{\Gamma}_{k-2}(a)\cap \mathcal{R}(\tilde{\Gamma}_{k-2})(a);B_k\cap^{k-1} _{j=k_0}B^c_j)\\&\geq \pr(B_k)\pr(\cap^{k-1} _{j=k_0}B^c_j)\pr(\Psi \in \tilde{\Gamma}_{k-2}(a)\cap \mathcal{R}(\tilde{\Gamma}_{k-2})(a))\\
&\geq \pr(B_k)\pr(\cap^{k-1} _{j=k_0}B^c_j)(1-2\epsilon),
 \end{split}
 \end{align}
where the second inequality in \eqref{algoaquiaux} is a consequence of the fact that the events $\{\Psi\in \mathcal{R}(\tilde{\Gamma}_{k})(i) \}$ and $\{\Psi\in \tilde{\Gamma}_{k}(i) \}$ have the same probability (see the comments above the statement of this lemma) and item $(i)$ of Lemma \ref{Lemak-pendentpercolation2}.  Therefore, we conclude that 
\begin{align*}
  \pr(\Psi \in \tilde{\Gamma}_{\textbf{x}-2}(a)\cap \mathcal{R}(\tilde{\Gamma}_{\textbf{x}-2})(a))
&\geq \pr(B_{k_0})(1-2\epsilon)+ \sum\limits^{\infty}_{k=k_0+1}\pr(B_k)\pr(\cap^{k-1} _{j=k_0}B^c_j)(1-2\epsilon)\\
     &=(1-2\epsilon)\pr(B_1) \sum\limits^{\infty}_{k=k_0}(1-\pr(B_1))^{k-k_0}=1-2\epsilon
\end{align*}
and the proof of the lemma is complete.

 \end{proof}
Before establishing the next proposition, we define the following processes    
$$
\eta^{A,B}_t=\{x:\zeta^{A,B}_t(x)=1\} \quad\text{ and }\quad  \chi^{A,B}_t=\{y:\zeta^{A,B}_t(y)=2\},
$$
where $A$ and $B$ are disjoint subsets of $\mathbb{Z}$. We observe that the random sets $\eta^{A,B}_t$ and $\chi^{A,B}_t$ are the sets of sites that are occupied, respectively, by  particles of type $1$ and $2$ at time $t$, where the initial configuration is $\mathds{1}_{A}+2\mathds{1}_{B}$. Also, we define the events 
\begin{align}\label{losD}
\begin{split}
&\chi=\chi(A,B)=\{\chi^{A,B}_t\cap [1,\infty)\neq \emptyset\text{ i.o}\}\quad\text{ and }\quad \eta=\eta(A,B)=\{ \eta^{A,B}_s\cap (-\infty,0]\neq \emptyset\text{ i.o}\},
\end{split}
\end{align}
 and we set $D=D(A,B)=\chi\cap \eta$.
\begin{proposition}\label{lallave}
Let $A$, $B$ be two  disjoint sets of $\mathbb{Z}$. For every finite set $E$, we have
\begin{align*}
\pr\left(D(A,B)\cap\{\exists \,\textbf{t}:\,\zeta^{A,B}_s\equiv \zeta^{\textbf{1},\textbf{2}}_s \text{ in } E,\, \forall\, s\geq \textbf{t}\}^c\right)=0.
\end{align*} 
\end{proposition}
\begin{proof}
During the proof, $\epsilon$ is an arbitrary positive number. For this $\epsilon$, we choose $M'$ as in \eqref{elv}. 
 
We begin by defining the following event  
 \begin{align}
 V=& \left\lbrace\begin{array}{l}\chi;\exists \, x \text{ and }t, \text{ such that }x\in\chi^{A,B}_{t}\cap [M',M],\\
0<t\leq M, \text{ and } (x,t)\text{ is expanding to the right}\end{array}\right\rbrace \label{primerevento}\\
 &\cap\left\lbrace\begin{array}{l}\eta; \exists \, y\text{ and }s, \text{ such that }x\in \eta^{A,B}_{s}\cap [-M,-M'],\\
0<s\leq M \text{ and }(y,s)\text{ is expanding to the left}\end{array}\right \rbrace\label{segundoevento}\\
 &\cap \{\Psi\in \tilde{\Gamma}_{\textbf{x}-2}(a)\cap \mathcal{R}(\tilde{\Gamma}_{\textbf{x}-2})(a)\},\label{tercerevento}
 \end{align}
where $M$ is a constant that will  be defined below and the random variable $\textbf{x}=\textbf{x}(k_0)$ is as in  Lemma \ref{otroauxiliar} with $k_0=\lceil 2 M/\hat{N}+M/\hat{K}\hat{N}\rceil$.
 
First, we will prove that we can take $M$ such that the probability of $V$ is close to the probability of $D$. Then, we will prove that for all configurations in $V$, we have that  $\zeta^{A,B}_t\equiv \zeta^{\textbf{1},\textbf{2}}_t$ in $E$, for all $t$ large enough. With these two ingredients it  will be easier to conclude the  proposition.  

Now, we take $M$ larger than $M'$ and satisfying that 
\begin{align}\label{++cacharra}
\begin{split}
0\leq &\pr\left(\left\lbrace\begin{array}{l}\chi;\exists \, x \text{ and }t, \text{ such that }x\in\chi^{A,B}_{t}\cap [M',\infty),\\
t>0, \text{ and } (x,t)\text{ is expanding to the right}\end{array}\right\rbrace\right)\\
&-\pr\left(\left\lbrace\begin{array}{l}\chi;\exists \, x \text{ and }t, \text{ such that }x\in\chi^{A,B}_{t}\cap [M',M],\\
0<t\leq M, \text{ and } (x,t)\text{ is expanding to the right}\end{array}\right\rbrace\right)\leq \epsilon
\end{split}
\end{align}
and 
\begin{align}\label{morecacharra}
\begin{split}
0\leq &\pr\left(\left\lbrace\begin{array}{l}\eta; \exists \, y\text{ and }s, \text{ such that }x\in \eta^{A,B}_{s}\cap [-\infty,-M'],\\
s>0 \text{ and }(y,s)\text{ is expanding to the left}\end{array}\right \rbrace\right)\\
&-\pr\left(\left\lbrace\begin{array}{l}\eta; \exists \, y\text{ and }s, \text{ such that }x\in \eta^{A,B}_{s}\cap [-M,-M'],\\
0<s\leq M \text{ and }(y,s)\text{ is expanding to the left}\end{array}\right \rbrace\right)\leq \epsilon.
\end{split}
\end{align}
We claim that 
\begin{claim}\label{+cacharra}
\begin{equation}\label{semeacabaronloslabel0}
\pr(\chi)=\pr\left(\left\lbrace\begin{array}{l}\chi;\exists \, x \text{ and }t, \text{ such that }x\in\chi^{A,B}_{t}\cap [M',\infty),\\
t>0, \text{ and } (x,t)\text{ is expanding to the right}\end{array}\right\rbrace\right),
\end{equation}
and
\begin{align}\label{semeacabaronloslabel}
\begin{split}
&\pr(\eta)=\pr\left(\left\lbrace\begin{array}{l}\eta; \exists \, y\text{ and }s, \text{ such that }x\in \eta^{A,B}_{s}\cap (-\infty,-M'],\\
s>0 \text{ and }(y,s)\text{ is expanding to the left}\end{array}\right \rbrace\right).
\end{split}
\end{align}

\end{claim}
\begin{proof}[Proof of the claim]

 As in Remark \ref{remulti}, we  denote by $\Xi^{C}_t$ the classic contact process restricted to $\mathbb{N}$ with initial configuration $C$, a subset of $\mathbb{N}$. Also, we denote by $\textbf{T}^{C}$ the time of extinction of the process $\Xi^{C}_t$. In \cite{Andjel}, it is proved in the nearest neighbor scenario that the classic contact process and the classic contact process restricted to $\mathbb{N}$ have the same critical rate of infection. For the case $R\geq 2$ this is also valid and it can be proved  using Corollary \ref{CorollaryAMPV} for $A=[4\alpha\hat{N}\hat{K},\infty)$. Therefore, if we take $\lambda>\lambda_c$, we have
\begin{equation*}
\pr(\textbf{T}^{\{1\}}=\infty)=\rho^{+}>0.
\end{equation*}
For $\hat{\epsilon}>0$, we take $N=N(\hat{\epsilon},M')$ satisfying \eqref{elncito} for $\hat{\epsilon}$   and  larger than $M'$.
Since equation \eqref{sextaequacaoahi} is also valid for the process $\Xi^{C}_t$ and $\textbf{T}^{C}$, for this $N$ we take $\textsf{t}$ such that
\begin{equation}\label{algoaqui}
\pr(\textbf{T}^{C}\geq \textsf{t};|\Xi^{C}_{\textsf{t}}|\geq 2 N)\geq \frac{\rho^{+}}{2},
\end{equation}
for every $C$ subset of $\mathbb{N}$.

Next, we define the following stopping time
\begin{align*}
t_1=\inf\{t>\textsf{t}:\,\chi^{A,B}_t\cap \mathbb{N}\neq \emptyset\},
\end{align*}
and inductively, for $i\geq 2$, we define $t_i$ as follows
\begin{align*}
t_i=\inf\{t>t_{i-1}+\textsf{t} :\chi^{A,B}_t\cap \mathbb{N}\neq \emptyset\}.
\end{align*}
Since the particles of type $2$ restricted to $[1,\infty)$ behave like the classic contact process, using the strong Markov property and \eqref{algoaqui} we obtain
\begin{equation}\label{maje}
\pr(|\chi^{A,B}_{t_i+\textsf{t}}\cap \mathbb{N}|<2 N|t_i<\infty) \leq \left(1-\frac{\rho^+}{2}\right).
\end{equation}
Therefore, using \eqref{maje} recursively and the strong Markov property we have 
\begin{equation}\label{aymama}
\pr(|\chi^{A,B}_{t_i+\textsf{t}}\cap \mathbb{N}|<2 N;\,t_i<\infty\,\forall \,i)=0.
\end{equation}
Moreover, observe that $
\chi=\{t_i<\infty\, \forall\, i\}
$, thus, by \eqref{aymama} we have   
\begin{align}\label{algoaqui2}
\begin{split}
\pr(\chi)=\pr(t_i<\infty\, \forall\, i)
&=\pr(\exists\, t>\textsf{t}:|\chi^{A,B}_{t}\cap \mathbb{N}|\geq 2 N; t_i<\infty\, \forall\, i)\\
&\leq\pr(\exists\, t>\textsf{t}:|\chi^{A,B}_{t}\cap [M',\infty)|\geq  N;\chi)\leq \pr(\chi),
\end{split}
\end{align}
where in the last equality in \eqref{algoaqui2} we have used the fact that $M'<N$. By the strong Markov property and \eqref{elncito} we have
$$
\pr\left(\begin{array}{l}\exists\, t>\textsf{t}:|\chi^{A,B}_{t}\cap [M',\infty)|\geq  N;\text{ and no point of }\\(\chi^{A,B}_{t}\cap [M',\infty))\times\{t\} \text{ is expanding to the right}
\end{array}\right)\leq \hat{\epsilon}.
$$
This equation and \eqref{algoaqui2} imply that
\begin{align}\label{algoutil3}
\begin{split}
0\leq &\pr(\chi)-\pr\left(\left\lbrace\begin{array}{l}\chi;\exists \, x \text{ and }t, \text{ such that }x\in\chi^{A,B}_{t}\cap [M',\infty),\\
t>0, \text{ and } (x,t)\text{ is expanding to the right}\end{array}\right\rbrace\right)\leq \hat{\epsilon}.
\end{split}
\end{align}
Since the probabilities in  \eqref{algoutil3} do not depend on $\hat{\epsilon}$, and $\hat{\epsilon}$ is arbitrary, we obtain \eqref{semeacabaronloslabel0}.
Due to the symmetry in the construction of the two-type contact process, the proof of \eqref{semeacabaronloslabel} is similar to the proof of \eqref{semeacabaronloslabel0}.
\end{proof}
Now, we observe that Claim \ref{+cacharra}, \eqref{++cacharra}, and \eqref{morecacharra} imply that for  our choice of $M$  we have
\begin{align*}
\begin{split}
0\leq &\pr(\chi)-\pr\left(\left\lbrace\begin{array}{l}\chi;\exists \, x \text{ and }t, \text{ such that }x\in\chi^{A,B}_{t}\cap [M',M],\\
0<t\leq M, \text{ and } (x,t)\text{ is expanding to the right}\end{array}\right\rbrace\right)\leq \epsilon
\end{split}
\end{align*}
and
\begin{align*}
\begin{split}
0\leq &\pr(\eta)-\pr\left(\left\lbrace\begin{array}{l}\eta; \exists \, y\text{ and }s, \text{ such that }x\in \eta^{A,B}_{s}\cap [-M,-M'],\\
0<s\leq M \text{ and }(y,s)\text{ is expanding to the left}\end{array}\right \rbrace\right)\leq \epsilon.
\end{split}
\end{align*}
Thus, we have that $M$ is  such that the probabilities of the events \eqref{primerevento} and \eqref{segundoevento} are closer to the probabilities of $\chi$ and $\eta$, respectively. Since the event \eqref{tercerevento} has probability larger than $1-2\epsilon$, by Lemma \ref{otroauxiliar} we have 
\begin{equation}\label{buenobueno}
\pr(D\cap V^c)\leq 4 \epsilon. 
\end{equation} 

 Next, we will prove that for all configurations in $V$ we have that  $\zeta^{A,B}_t\equiv \zeta^{\textbf{1},\textbf{2}}_t$ in $E$, for all $t$ large enough.
Observe that, for the configurations in  the event \eqref{primerevento}, there exists a point $(x,t)$ expanding to the right. Therefore, there exists a path in $\Psi$, which we denote by $\gamma$. We denote  by $(i,k)$ the initial point of $\gamma$.  We identify the path $\gamma$ with  a sequence $\{m_j,j\}$ that satisfies $m_j\geq (j-k)/2+i+1 $. We note that by the definition of the event \eqref{primerevento},  we have that $v \leq i\leq M/\hat{N}+2$ and since in  the event \eqref{primerevento} $t\leq M$, it holds $k\leq M/\hat{K}\hat{N}$.

  The event \eqref{segundoevento} implies the existence of  a point $(y,s)$ expanding to the left, and this gives  a path $\beta$ in $\Psi$. We denote by $(\hat{i},\hat{k})$ the initial point of $\beta$, and similar to the path $\gamma$, we have that $- M/\hat{N}-2\leq \hat{i}\leq -v$ and $\hat{k}\leq M/\hat{N}\hat{K}$. We identify the path $\beta$ with the sequence $\{\hat{m}_j, j\}$ that satisfies $\hat{m}_j\leq-(j-\hat{k})/2+\hat{i}+1$.  
  
 The  event \eqref{tercerevento} implies that there exists a path in $\Psi$ that connects $[(\textbf{x}-2)/2+a,\infty)\times\{0\}$ with $(\imath, \textbf{x}-2)$ ($\imath$ depending on $a$ and $\textbf{x}-2$), which we denote by $\hat{\gamma}$.  This path does not intersect the set $\{(m,j): m\leq (\textbf{x}-2)/2-\frac{j}{2}+a\}$. By our choice of $k_0$ and the properties of $i$ and $k$, we have 
  \begin{align*}
 &k\leq M/\hat{K}\hat{N} \leq  k_0-2 \leq \textbf{x}-2,\\
 &i \leq \frac{M}{\hat{N}}+2 \leq\frac{\textbf{x}}{2}-\frac{M}{2\hat{K}\hat{N}}+2\leq  \frac{(\textbf{x}-2)}{2}-\frac{k}{2}+3\leq \frac{(\textbf{x}-2)}{2}-\frac{k}{2}+a\leq \hat{\gamma}(k),\\
 &m_{\textbf{x}-2}\geq i+\frac{(\textbf{x}-2)}{2}-\frac{k}{2} \geq i+\frac{ M}{\hat{N}}\geq i+v+2 \geq a+2.
\end{align*}   
Therefore,   we have that $\hat{\gamma}$ intersects the path $\gamma$, and we denote by $k_1$ the time of the intersection. In Figure \ref{Figura1}, we represent these paths to clarify the definitions. The union of the renormalized boxes that correspond to the part of the path $\hat{\gamma}$ connecting $(m_{k_1},k_1)$ with $(\imath, \textbf{x}-2)$ is denoted by $\textbf{R}_1$. We denote by $\textbf{R}_2$ the union of the renormalized boxes corresponding to the infinite portion of the path $\gamma$ starting at the point $(m_{k_1},k_1)$.

Similarly, by the definition of the event \eqref{tercerevento}, there exists a path $\hat{\beta}$ that connects $(-\infty,-a-(\textbf{x}-2)/2]\times\{0\}$ with $(-\imath,\textbf{x}-2)$, and this path does not intersect the set $\{(m,s): m\geq -a-(\textbf{x}-2)/2+\frac{s}{2}\}$. By similar arguments to those used with the paths $\gamma$ and $\hat{\gamma}$, we conclude that the path $\beta$ intersects the path $\hat{\beta}$, and we denote by $k_2$ the time of the intersection.  We denote by $\textbf{B}_1$ the union of the renormalized boxes that correspond to the portion of the path $\hat{\beta}$ connecting $(\hat{m}_{k_2},k_2)$ with $(-\imath, \textbf{x}-2)$. Also, we denote by $\textbf{B}_2$ the union of the renormalized boxes that correspond to the infinite portion of the path $\beta$ starting at the point $(\hat{m}_{k_2},k_2)$.
\begin{figure}[h]
\begin{center}
  \begin{overpic}[scale=0.3,unit=1mm]{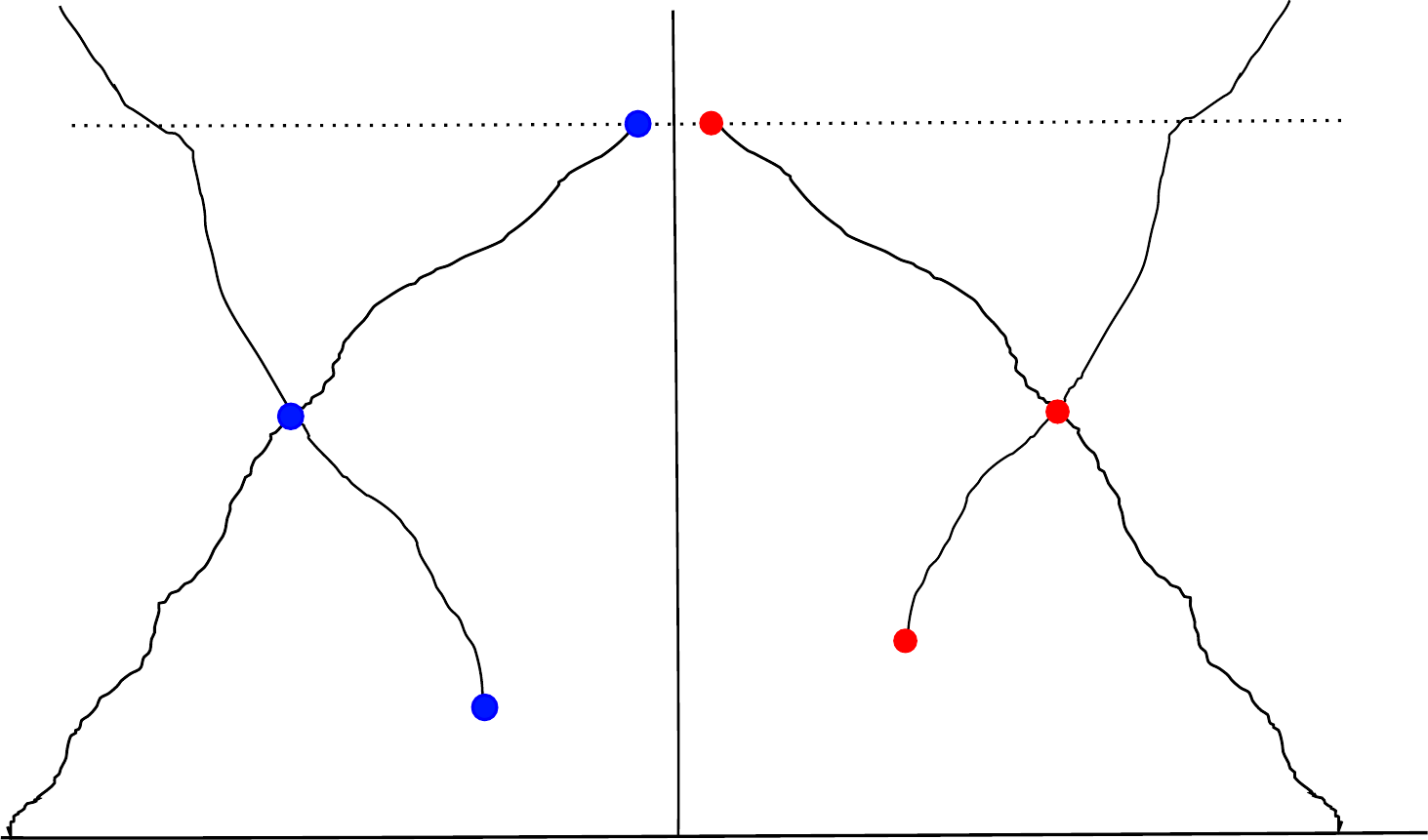}
   \put(-4,21.5){\parbox{0.4\linewidth}{
\begin{tiny}
$\textbf{x}-2$
\end{tiny}}}
\put(37,23){\parbox{0.4\linewidth}{
\begin{tiny}
$\gamma$
\end{tiny}}}
\put(38,2){\parbox{0.4\linewidth}{
\begin{tiny}
$\hat{\gamma}$
\end{tiny}}}
\put(33,12){\parbox{0.4\linewidth}{
\begin{tiny}
$(m_{k_1},k_1)$
\end{tiny}}}
\put(23,5){\parbox{0.4\linewidth}{
\begin{tiny}
$(i,k)$
\end{tiny}}}
\put(22,22.5){\parbox{0.4\linewidth}{
\begin{tiny}
$\imath$
\end{tiny}}}
\put(17,22.5){\parbox{0.4\linewidth}{
\begin{tiny}
$-\imath$
\end{tiny}}}
\put(4,23){\parbox{0.4\linewidth}{
\begin{tiny}
$\beta$
\end{tiny}}}
\put(2.5,2){\parbox{0.4\linewidth}{
\begin{tiny}
$\hat{\beta}$
\end{tiny}}}
\put(10,3){\parbox{0.4\linewidth}{
\begin{tiny}
$(\hat{i},\hat{k})$
\end{tiny}}}
\put(9,12){\parbox{0.4\linewidth}{
\begin{tiny}
$(\hat{m}_{k_2},k_2)$
\end{tiny}}}
\end{overpic}
\caption{ Representation of the paths $\gamma$, $\beta$, $\hat{\gamma}$ and $\hat{\beta}$.}\label{Figura1}
\end{center}
\end{figure}
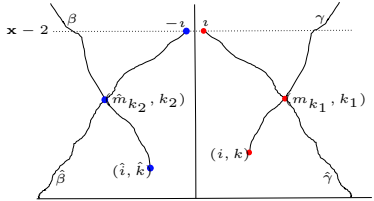

By the Mountford-Sweet renormalization all the points in $\textbf{R}_1\cup \textbf{R}_2$ connected in the Harris graph to $\mathbb{Z}\times\{0\}$, are connected within $[0,\infty)$ to $I^{\hat{N},\hat{K}}_{(m_{k_1},k_1)}$. Observe that all the connected points in $I^{\hat{N},\hat{K}}_{(m_{k_1},k_1)}$ also are connected to $I^{\hat{N},\hat{K}}_{(i,k)}$ inside $[0,\infty)$. Since $(x,t)$ is expanding to the right,  all the connected points in $I^{\hat{N},\hat{K}}_{(i,k)}$ are connected to $(x,t)$ inside $[0,\infty)$. 
Event \eqref{primerevento} also gives that $(x,t)$ is occupied by a particle of type $2$ for the process $\{\zeta^{A,B}_t\}$. Therefore, all the occupied points in $I^{\hat{N},\hat{K}}_{(i,k)}$, and consequently in $I^{\hat{N},\hat{K}}_{(m_{k_1},k_1)}$, are of type $2$ for the process with initial configuration $\mathds{1}_A+2\mathds{1}_B$. 
On the other hand, using the boxes in the path $\hat{\gamma}$, we have that all the points in $I^{\hat{N},\hat{K}}_{(m_{k_1},k_1)}$ connected in the Harris graph to $\mathbb{Z}\times\{0\}$  are connected to $[\hat{N}\{a+(\textbf{x}-2)/2\}, \infty)\times \{0\}$ inside $[0,\infty)$. Therefore, the connected points in  $I^{\hat{N},\hat{K}}_{(m_{k_1},k_1)}$ are also occupied by particles of type $2$ for the process $\{\zeta^{\textbf{1}, \textbf{2}}_t\}$. Thus, all the connected points in $\textbf{R}_1\cup \textbf{R}_2$ are occupied by particles of type $2$ for both processes. By similar arguments, we have that both processes are equal in the set $\textbf{B}_1\cup \textbf{B}_2$.

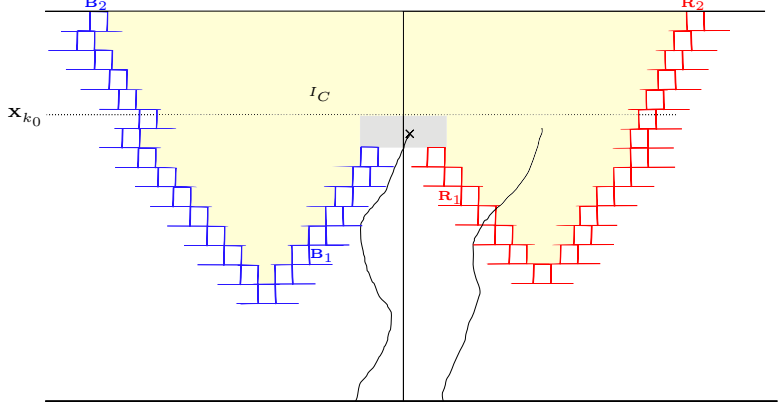
\begin{figure}[h]
\begin{center}
  \begin{overpic}[scale=0.4,unit=1mm]{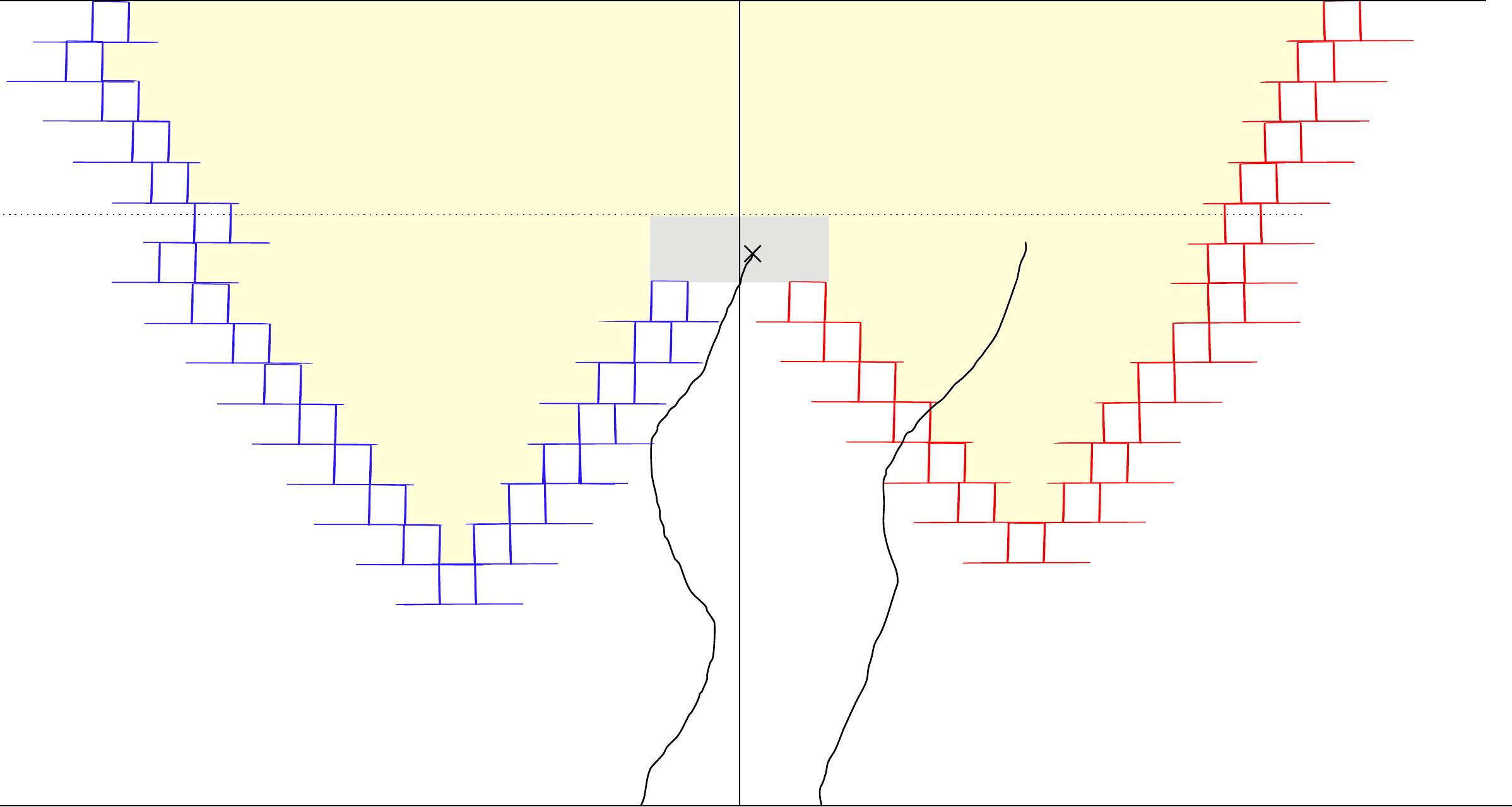}
    \put(-5,37){\parbox{0.4\linewidth}{
\begin{tiny}
$\textbf{X}_{k_0}$
\end{tiny}}}
    \put(5.2,52){\parbox{0.4\linewidth}{
\begin{tiny}
\textcolor{blue}{$\textbf{B}_{2}$}
\end{tiny}}}
    \put(35.2,18.5){\parbox{0.4\linewidth}{
\begin{tiny}
\textcolor{blue}{$\textbf{B}_{1}$}
\end{tiny}}}
    \put(52.2,26.2){\parbox{0.4\linewidth}{
\begin{tiny}
\textcolor{red}{$\textbf{R}_{1}$}
\end{tiny}}}
    \put(84.5,52){\parbox{0.4\linewidth}{
\begin{tiny}
\textcolor{red}{$\textbf{R}_{2}$}
\end{tiny}}}
   \put(35,40){\parbox{0.4\linewidth}{
\begin{tiny}
$I_C$
\end{tiny}}}
\end{overpic}
\caption{ We have represented the region $I_C$ in light yellow. The gray rectangle represents that every path that ends in $I_C$ can not intersect the rectangle $[-M',M']\times[\textbf{X}_{k_0}-\hat{K}\hat{N},\textbf{X}_{k_0}]$.}\label{Figura2}
\end{center}
\end{figure}
Now, we define the set
$$
C=\textbf{B}_2\cup \textbf{B}_1\cup( [-M',M']\times[\textbf{X}_{k_0}-\hat{K}\hat{N},\textbf{X}_{k_0}])\cup \textbf{R}_1\cup \textbf{R}_2.
$$
In  Figure \ref{Figura2}, we represent the set $C$.  Observe that the base of the rectangle $[-M',M']\times[\textbf{X}_{k_0}-\hat{K}\hat{N},\textbf{X}_{k_0}]$  intersects the top of the renormalized boxes $(\imath,\textbf{x}-2)$ and $(-\imath,\textbf{x}-2)$. Since these boxes  are subsets of $\textbf{R}_1$ and $\textbf{B}_1$, respectively, the set $C$ is connected.  Also, observe that the complement of the set $C$ has two connected components in $\mathbb{R}\times [0,\infty)$. We call the \emph{inside of $C$} the connected component that does not have the $(0,0)$,  and we denote it by $I_C$.   All the sets whose union  define  $C$ have a width larger than $R$, therefore every path in the Harris graph that connects $\mathbb{Z}\times \{0\}$ with $I_C$ intersects $C$. We observe  that by our definition of $\textbf{X}_{k_0}$, every path that connects $\mathbb{Z}\times \{0\}$ with $I_C$ can not intersect the rectangle $[-M',M']\times[\textbf{X}_{k_0}-\hat{K}\hat{N},\textbf{X}_{k_0}]$. Hence, each of these paths intersects the sets $\textbf{B}_2\cup \textbf{B}_1$ or $\textbf{R}_1\cup \textbf{R}_2$, and in these two sets, the processes $\{\zeta^{A,B}_t\}$ and $\{\zeta^{\textbf{1}, \textbf{2}}_t\}$ are equal. Thus, these two processes are also equal  in  $I_C$.

It remains to choose a time $\textbf{t}$ such that $E\times[\textbf{t},\infty)$ is a subset of $I_C$. For this purpose, we take $\textbf{t}=\max \{M+4\hat{K}\max E;M-4\hat{K}\min E;\textbf{X}_{k_0}\}+\hat{K}\hat{N}$, and we define $k(s)=\lfloor s/\hat{K}\hat{N}\rfloor$. For every $s\geq \textbf{t}$ we have 
\begin{align*}
&m_{k(s)}\geq m_{k(\textbf{t})}\geq i+1+\frac{k(\textbf{t})-k}{2}\geq i+ \frac{M+4\hat{K}\max E}{2\hat{K}\hat{N}}-\frac{k}{2}\\
&=i+\frac{M}{2\hat{K}\hat{N}}-\frac{k}{2}+\frac{2 \max E}{\hat{N}}\geq \frac{2 \max E}{\hat{N}}\Rightarrow\frac{\hat{N}m_{k(s)}}{2}\geq \max E.
\end{align*}
On the other hand, we have 
\begin{align*}
&\hat{m}_{k(s)}\leq \hat{m}_{k(\textbf{t})}\leq \hat{i}-1-\frac{k(\textbf{t})-\hat{k}}{2}\leq \hat{i}- \frac{M-4\hat{K}\min E}{2\hat{K}\hat{N}}+\frac{\hat{k}}{2}\\
&=\hat{i}-\frac{M}{2\hat{K}\hat{N}}+\frac{\hat{k}}{2}+\frac{2 \min E}{\hat{N}}\leq \frac{2 \min E}{\hat{N}}\Rightarrow\frac{\hat{N}\hat{m}_{k(s)}}{2}\leq \min E.
\end{align*}
The set $E\times\{s\}$ is between the renormalized boxes $(\hat{m}_{k(s)},k(s))$ and $(m_{k(s)},k(s))$, for all $s\geq \textbf{t}$. Observe that these boxes are in the set $C$. Since we also have that $\textbf{t}\geq \textbf{X}_{k_0}$, $E\times\{s\}\subset I_{A}$ for all $s\geq \textbf{t}$. Thus, we have proved

\begin{equation*}
V\subset\{\exists\, \textbf{t}:\zeta^{\textbf{1},\textbf{2}}_s\equiv\zeta^{\zeta_0}_s \text{ in }E\, \forall\, s\geq \textbf{t} \}.
\end{equation*}
The above inclusion and  \eqref{buenobueno} imply 
 \begin{align}\label{buenoesto}
 \pr(D\cap \{\exists\, \textbf{t}:\zeta^{\textbf{1},\textbf{2}}_s\equiv\zeta^{\zeta_0}_s \text{ in }E\, \forall\, s\geq \textbf{t} \}^{c})\leq \pr(D\cap V^{c})\leq 4 \epsilon.
 \end{align}
 Observe that the first probability in \eqref{buenoesto} does not depend on $\epsilon$. Since $\epsilon$ is arbitrary, this probability is zero, and we conclude the proof. 
 \end{proof}

\begin{remark}
The proof of Proposition \ref{lallave} is valid for the contact process with any finite range, even for $R=1$. But in the case $R=1$, the proof became simpler since we do not need the Mountford-Sweet renormalization. Now, we explain the simplifications in the proof for the nearest-neighbor scenario. In this case, we can simplify the definition of expanding as follow, a point $(x,t)$ is \emph{expanding to the right} \emph{(left)} if there is an infinite path to the right (left) of the half-line $y=\frac{\alpha}{2}s+x$ ($y=-\frac{\alpha}{2}s+x$) for $s\geq t$. Proposition \ref{AMPV} and Corollary \ref{CorollaryAMPV} are also valid for this definition of expanding. The first changes in the proof of the proposition is that in the definition of the set $V$ we take $M'=N$ as in \eqref{elncito} and $M$ satisfying \eqref{++cacharra} and \eqref{morecacharra}. Also, we redefine $B_k$ taking $\hat{K}\hat{N}=1$ and $R=1$, and in the definition of the stopping time $\textbf{X}=\textbf{X}_{k_0}$ we take, $k_0= M+\frac{2}{\alpha}M$. Moreover, we change the event \eqref{tercerevento} in the definition of $V$ by the following event
\begin{equation}
\left\lbrace\begin{array}{c}
[1,\infty)\times\{0\}\rightarrow[1,N]\times\{\textbf{X}-1\}\text{ with a path to the right of the line }y=-\frac{\alpha}{2}(s-\textbf{X})\\\text{and }
(-\infty,0]\times\{0\}\rightarrow[-N,0]\times\{\textbf{X}-1\}\text{ with a path to the left of the line }y=\frac{\alpha}{2}(s-\textbf{X})
\end{array}\right\rbrace.
\end{equation}
Using the duality of the Harris construction, \eqref{elncito} and the same ideas in the proof of Lemma \ref{otroauxiliar}, it is possible to conclude that this last event has probability larger than $1-2\epsilon$. Therefore, the probability of the event $V$ is close to the probability of $D$.

The argument for concluding that both processes are equal in a large region $I_C$  is very similar to the one in the proof of the proposition. For the configurations in the event $V$, we have two pairs of paths that intersect each other, one pair in the half-plane $(-\infty,0]\times[0,\infty)$ and the other in $[1,\infty)\times[0,\infty)$. Figure \ref{Figure3} illustrate these paths. The gray rectangle in the middle is the region where no path crosses and comes from the definition of the stopping time $\textbf{X}$. As in the proof of the proposition, the points where the pair of paths intersect are blue and red respectively for both processes, the one with initial configuration $\mathds{1}_A+2\mathds{1}_B$ and the one with initial configuration $\mathds{1}_{(-\infty,0]}+2\mathds{1}_{[1,\infty)}$. Then, if a path ends in the yellow region in the figure, it must cut one of the four paths, and since both processes are equal in these paths, there will be equal also in the yellow region, $I_C$.  The rest of the proof follows as above. 
\begin{figure}[h]
\begin{center}
  \begin{overpic}[scale=0.3,unit=1mm]{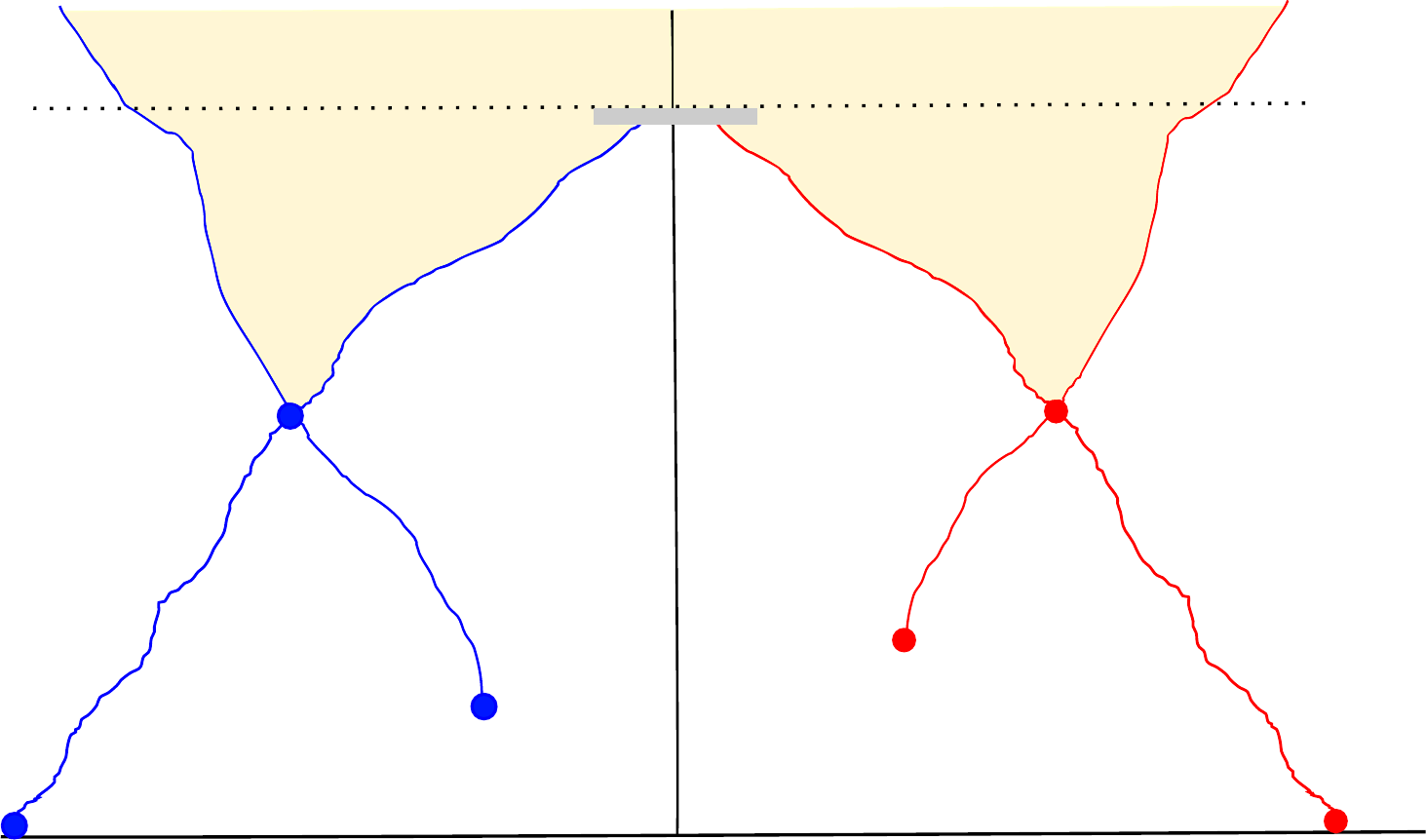}
   \put(-1,22){\parbox{0.4\linewidth}{
\begin{tiny}
$\textbf{X}$
\end{tiny}}}
\end{overpic}
\end{center}
\caption{The yellow region is the region $I_C$ where the two processes coincide. The gray rectangle is the region where no path cross.}\label{Figure3}
\end{figure}
\end{remark}
\begin{proof}[Proof of Theorem \ref{teoconvergence}]
 We take a set $F\in \mathcal{F}$ depending on a finite number of sites in $\mathbb{Z}$, and we denote $E=E(F)$ the set of those sites. We remember that the measure $\nu$ is supported in the set $\mathcal{A}$ defined in \eqref{defmathcalA}.  Since the configurations in $\mathcal{A}$  have infinitely many particles of type $1$ in $(-\infty,0]$  and infinitely many particles of type $2$ in $[1,\infty)$, for the process with initial configuration in $\mathcal{A}$ there will be a particle of type $1$ in $(-\infty,0]$ and a particle of type $2$ in $[1,\infty)$ for all times. This is 
\begin{equation}\label{vamosmari}
\pr(\chi^{\zeta_0}_t\cap [1,\infty)\neq \emptyset\text{ i.o};  \eta^{\zeta_0}_s\cap (-\infty,0]\neq \emptyset\text{ i.o}
)=\pr(D(\zeta_0))=1
\end{equation} 
for all $\zeta_0\in \mathcal{A}$. Equation \eqref{vamosmari} and Proposition \ref{lallave} imply
\begin{equation}\label{vamosmari2}
\pr(\exists \,\textbf{t}:\,\zeta^{\zeta_0}_s\equiv \zeta^{\textbf{1},\textbf{2}}_s \text{ in } E,\, \forall\, s\geq \textbf{t})=1
\end{equation}
for all $\zeta_0\in \mathcal{A}$.   Observe the following calculations 
\begin{align}\label{blubliblu*}
\begin{split}
\lim \limits_{t\rightarrow \infty}\pr(\zeta^{\textbf{1},\textbf{2}}_t\in F)&=\lim\limits_{t\rightarrow \infty}\int_{\zeta_0\in \mathcal{A}}\pr(\zeta^{\textbf{1},\textbf{2}}_{t}\in F)d\nu(\zeta_0)\\
&=\lim\limits_{t\rightarrow \infty}\int_{\zeta_0\in \mathcal{A}}\pr(\zeta^{\textbf{1},\textbf{2}}_{t}\in F;\exists \,\textbf{t}:\,\zeta^{\zeta_0}_s\equiv \zeta^{\textbf{1},\textbf{2}}_s \text{ in } E,\, \forall\, s\geq \textbf{t})d\nu(\zeta_0)\\
&=\int_{\zeta_0\in \mathcal{A}}\lim\limits_{t\rightarrow \infty}\pr(\zeta^{\textbf{1},\textbf{2}}_{t}\in F;\exists \,\textbf{t}:\,\zeta^{\zeta_0}_s\equiv \zeta^{\textbf{1},\textbf{2}}_s \text{ in } E,\, \forall\, s\geq \textbf{t})d\nu(\zeta_0)\\
&=\int_{\zeta_0\in \mathcal{A}}\lim\limits_{t\rightarrow \infty}\pr(\zeta^{\zeta_0}_{t}\in F;\exists \,\textbf{t}:\,\zeta^{\zeta_0}_s\equiv \zeta^{\textbf{1},\textbf{2}}_s \text{ in } E,\, \forall\, s\geq \textbf{t})d\nu(\zeta_0)\\
&=\lim\limits_{t\rightarrow \infty}\int_{\zeta_0\in \mathcal{A}}\pr(\zeta^{\zeta_0}_{t}\in F;\exists \,\textbf{t}:\,\zeta^{\zeta_0}_s\equiv \zeta^{\textbf{1},\textbf{2}}_s \text{ in } E,\, \forall\, s\geq \textbf{t})d\nu(\zeta_0)\\
&=\lim\limits_{t\rightarrow \infty}\int_{\zeta_0\in \mathcal{A}}\pr(\zeta^{\zeta_0}_{t}\in F)d\nu(\zeta_0)=\nu(F).
\end{split}
\end{align}
In the second and sixth equalities of \eqref{blubliblu*} we have used \eqref{vamosmari2}, and in the third and fifth equalities, we have used the  Dominated Convergence Theorem. Since $F$ is any finite dimensional set we have proved that $\zeta^{\textbf{1},\textbf{2}}_t$ converges in distribution to $\nu$.  
\end{proof}
\begin{corollary}\label{el que pensé que no usaba}
For every $\zeta_0\in \mathcal{A}$ we have that
\begin{equation}\label{yacasiportuvida}
\zeta^{\zeta_0}_t\underset{t\rightarrow\infty}{\longrightarrow}\nu \text{ in Distribution.}
\end{equation}
\end{corollary}
\begin{proof}
Let $F$ be a finite dimensional set and $E$ the set of sites on which the elements of $F$ depend. Equation \eqref{vamosmari2} implies that
$$
\lim\limits_{t\rightarrow\infty}\pr(\zeta^{\zeta_0}_t\in F)=\lim\limits_{t\rightarrow\infty}\pr(\zeta^{\textbf{1},\textbf{2}}_t\in F),
$$
which, combined with  Theorem \ref{teoconvergence}, gives  \eqref{yacasiportuvida}.
\end{proof}
\begin{corollary}There exist two positive constant $\hat{c}$ and $\hat{C}$ such that
\begin{equation}\label{otrodetallito}
\nu(\zeta: \textbf{r}^{1}(\zeta)-\textbf{l}^{2}(\zeta)\geq N)\leq \hat{C}e^{-\hat{c}N},
\end{equation}
for all $N$.
\end{corollary}
\begin{proof}
Proposition \ref{tightness} implies that 
$$
\pr(\textbf{r}^{1}_t\geq N/2)\leq Ce^{-\frac{c}{2M}N}\quad\text{ and }\quad \pr(\textbf{l}^{2}_t\leq N/2)\leq Ce^{-\frac{c}{2M}N},
$$
for $t\geq \hat{K}\hat{N}$ and for all $N$. Therefore 
\begin{equation}\label{undetallito}
\pr(\textbf{r}^{1}_t-\textbf{l}^2_t\geq N)\leq 2 Ce^{-\frac{c}{2M}N}.
\end{equation}
By  Theorem \ref{teoconvergence}, the right-hand side of \eqref{undetallito} converges to the right-hand side of \eqref{otrodetallito} when $t$ goes to infinity. Thus, we have
$$
\nu(\zeta: \textbf{r}^{1}(\zeta)-\textbf{l}^{2}(\zeta)\geq N)\leq 2 Ce^{-\frac{c}{2M}N},
$$
for all $N$, and we conclude the proof of the corollary. 
\end{proof}
 \section{Proof of Theorem 2}\label{Proof theo 2}
 Let $A$ and $B$ be two disjoint subsets of $\mathbb{Z}$. Also, consider $F$ a subset of $\{0,1,2\}^{\mathbb{Z}}$ depending on a finite number of coordinates. We denote by $E=E(F)$ the set of coordinates on which $F$ depends. We define the measure $\mu_1$ as the limit in distribution of $\{\zeta^{\mathbb{Z},\emptyset}_t\}$. The measure $\mu_1$ is supported in $\{\zeta\in \{0,1,2\}^ {\mathbb{Z}}: \, \zeta(x)\neq 2\, \forall\, x \in \mathbb{Z}\}$, and it is essentially the non-trivial invariant measure for the classic contact process in $\{0,1\}^{\mathbb{Z}}$. Similarly, we define the measure $\mu_2$ as the limit in distribution of $\{\zeta^{\emptyset, \mathbb{Z}}_t\}$, which is supported in  $\{\zeta\in \{0,1,2\}^ {\mathbb{Z}}: \, \zeta(x)\neq 1\, \forall\, x \in \mathbb{Z}\}$, and it is also basically the non-trivial invariant measure for the contact process in $\{0,2\}^{\mathbb{Z}}$.
 Also, we define the times of extinction of each type of particles for the two-type process with initial configuration $\mathds{1}_{A}+2\mathds{1}_B$, as follows
 \begin{equation*}
 \tau^{A,B}_1=\inf\{t: \eta^{A,B}_t=\emptyset\}\quad\text{ and }\quad \tau^{A,B}_2=\inf\{t: \chi^{A,B}_t=\emptyset\},
 \end{equation*}
 and we define $\tau^{A,B}=\min\{\tau^{A,B}_1,\tau^{A,B}_2\}$. 

  We divide the proof of Theorem \ref{Teo2} into three lemmas. In Lemma \ref{lemmaequacaoahi} below, we prove that 
 \begin{equation}\label{Formula1}
 \lim\limits_{t\rightarrow \infty}\pr(\zeta^{A,B}_t\in F;\tau^{A,B}_1=\infty;\tau^{A,B}_2<\infty)=\mu_1(F)\pr(\tau^{A,B}_1=\infty;\tau^{A,B}_2<\infty).\end{equation}
This limit gives that if the particles of type  $2$ die out, the process converges to the non-trivial invariant measure for the classic contact process. By the symmetry of our construction, we have the analogous limit if the particles of type $1$ die out. This is
  \begin{equation*}\lim\limits_{t\rightarrow \infty}\pr(\zeta^{A,B}_t\in F;\tau^{A,B}_2=\infty;\tau^{A,B}_1<\infty)=\mu_2(F)\pr(\tau^{A,B}_2=\infty;\tau^{A,B}_1<\infty).\end{equation*}
 Also, it is trivial to see that when the two types of particles die out we have
\begin{equation*}\lim\limits_{t\rightarrow \infty}\pr(\zeta^{A,B}_t\in F;\tau^{A,B}_1<\infty;\tau^{A,B}_2<\infty)=\delta_{\emptyset}(F)\pr(\tau^{A,B}_1<\infty;\tau^{A,B}_2<\infty).
    \end{equation*}
    
Next, we study what happens when both particles survive for all times. First, in Lemma \ref{lemmadummy}, we consider the case when after a random time there is no particle of type $1$ in $(-\infty,0]$, where type $1$ particles  have priority. In this case, after a random time, the two-type process behaves as a Grass-Bushes-Tree process, where the bushes are the particles of type $1$, and the trees are the particles of type $2$. Therefore, the two-type process converges to the measure $\mu_2$, which is supported in the configuration without particles of type $1$. More precisely,  we will prove that
\begin{align}\label{Formula2}\begin{split}
    &\lim\limits_{t\rightarrow \infty}\pr(\zeta^{A,B}_t\in F;\tau^{A,B}=\infty;\exists \,t'\, \forall s\geq t'\eta^{A,B}_s\cap (-\infty,0]=\emptyset)\\
    &=\mu_2(F)\pr(\tau^{A,B}=\infty;\exists\,t'\, \forall s\geq t'\eta^{A,B}_s\cap (-\infty,0]=\emptyset).
    \end{split}
\end{align}
By the symmetry of our construction, when the particles of type  $2$  only survive in $(-\infty,0]$, we have the analogous limit
\begin{align*}\begin{split}
    &\lim\limits_{t\rightarrow \infty}\pr(\zeta^{A,B}_t\in F; \tau^{A,B}=\infty;\exists t'\,\forall s\geq t'\chi^{A,B}_{s}\cap[1,\infty)= \emptyset)\\
    &=\mu_1(F)\pr(\tau^{A,B}=\infty;\exists t'\,\forall s\geq t'\chi^{A,B}_{s}\cap[1,\infty)= \emptyset).
    \end{split}
\end{align*}

Finally, in Lemma \ref{ultimo} we study the case when both types of particles survive and for infinitely large times there are particles of type $1$ in $(-\infty,0]$ and particles of type $2$ in $[1,\infty)$. Specifically,  we will obtain
\begin{align}\label{Formula3}\begin{split}
    &\lim\limits_{t\rightarrow \infty}\pr(\zeta^{A,B}_t\in F; \tau^{A,B}=\infty;\eta^{A,B}_s\cap (-\infty,0]\neq\emptyset\text{ i.o};\chi^{A,B}_{s}\cap[1,\infty)\neq \emptyset\text{ i.o})\\
    &=\nu(F)\pr(\tau^{A,B}=\infty;\eta^{A,B}_s\cap (-\infty,0]\neq\emptyset\text{ i.o};\chi^{A,B}_{s}\cap[1,\infty)\neq \emptyset\text{ i.o}).
    \end{split}
\end{align}

We have covered all the possibilities for the survival or extinction of the two types of particles. Therefore, for an arbitrary finite dimensional set $F$,  $\pr(\zeta^{A,B}_t\in F)$ converges to a convex combination of $\mu_1(F),\,\mu_2(F),\delta_{\emptyset}(F)$ and $\nu(F)$. This is sufficient to obtain Theorem \ref{Teo2}.
\begin{lemma}\label{lemmaequacaoahi}
Let $A$ and $B$ be two disjoint subsets of $\mathbb{Z}$ and let $F$ be a finite dimensional set in $\mathcal{F}$, then \eqref{Formula1} holds.
\end{lemma}
\begin{proof}Observe that 
\begin{align}\begin{split}\label{quatroequacaoahi}
&\lim\limits_{t\rightarrow \infty}\pr(\zeta^{A,B}_t\in F;\tau^{A,B}_1=\infty;\tau^{A,B}_2<\infty)=\lim\limits_{t\rightarrow \infty}\pr(\zeta^{A,B}_t\in F;\tau^{A,B}_1=\infty;\tau^{A,B}_2<t)\\
&=\lim\limits_{t\rightarrow \infty}\pr(\xi^{A\cup B}_t\in F_1;T^{A\cup B}=\infty;\tau^{A,B}_2<t)=\lim\limits_{t\rightarrow \infty}\pr(\xi^{A\cup B}_t\in F_1;T^{A\cup B}=\infty;\tau^{A,B}_2<\infty),
\end{split}
\end{align}where $F_1$ are all the configurations in $F$ that do not have particles of type $2$. The second equality in \eqref{quatroequacaoahi} follows from the fact that if the particles of type $2$ die out, then the process with two types of particles behaves like the classic contact process. Next, we will prove that 
\begin{align}\begin{split}\label{quintaequacaoahi}
    &\lim\limits_{t\rightarrow \infty}\pr(\xi^{A\cup B}_t\in F_1;T^{A\cup B}=\infty;\tau^{A,B}_2<\infty)=\lim\limits_{t\rightarrow \infty}\pr(\xi^{\mathbb{Z}}_t\in F_1;T^{A\cup B}=\infty;\tau^{A,B}_2<\infty).
    \end{split}
\end{align}
 The limit \eqref{quintaequacaoahi} may be proved in much the same way as Proposition \ref{lallave}. Therefore, we give only the main ideas of the proof. 
 For $\epsilon$ arbitrary,  we choose $N$ as in Proposition \ref{AMPV} and by the strong Markov property, we have
 \begin{equation}\label{decimaequacaoahi}
    \pr(\exists\,s:|\xi^{A\cup B}_s|\geq N;\nexists\, (x,s) \text{ expanding and }x\in \xi^{A\cup B}_s)\leq \epsilon. 
\end{equation}
We use \eqref{sextaequacaoahi} and \eqref{decimaequacaoahi} to obtain
\begin{align*}
&\pr(T^{A\cup B}=\infty)-\pr(T^{A\cup B}=\infty;\exists\,(x,s)\text{ expanding and }x\in \xi^{C}_s)\leq \epsilon,
\end{align*}
therefore, we have
\begin{align}\label{ay}
\begin{split}
&|\pr(\xi^{A\cup B}_t\in F_1;T^{A\cup B}=\infty;\tau^{A,B}_2<\infty)-\pr(\xi^{\mathbb{Z}}_t\in F_1;T^{A\cup B}=\infty;\tau^{A,B}_2<\infty)|\\
&\leq \pr(\xi^{A\cup B}_t\not \equiv \xi^{\mathbb{Z}}_t\text{ in }E;T^{A\cup B}=\infty)\\
&\leq\pr(\xi^{A\cup B}_t\not \equiv \xi^{\mathbb{Z}}_t\text{ in }E;T^{A\cup B}=\infty;\exists\,(x,s)\text{ expanding and }x\in \xi^{A\cup B}_s)+\epsilon.
\end{split}
\end{align}
If we take $t$ large enough such that $E\times\{t\}$ is inside the descendancy barrier of the expanding point $(x,s)$, then $\xi^{C}_t$ is equal to  $\xi^{\mathbb{Z}}_t$ in $E$. Therefore, the probability in the last inequality in \eqref{ay} converges to zero. Since $\epsilon$ is arbitrary, we obtain \eqref{quintaequacaoahi}.

To conclude the proof we observe that the limit
\begin{align*}
    \lim\limits_{t\rightarrow \infty}\pr(\xi^{\mathbb{Z}}_t\in F_1;T^{A\cup B}=\infty;\tau^{A,B}_2<\infty)&=\mu_1(F_1)\pr(T^{A\cup B}=\infty;\tau^{A,B}_2<\infty)\\
    &=\mu_1(F_1)\pr(\tau^{A,B}_1=\infty;\tau^{A,B}_2<\infty)
\end{align*}
follows from the same arguments used in the proof of Theorem $2.28$ page $284$ in \cite{Liggett} for the case when $R=1$.
\end{proof}

\begin{lemma}\label{lemmadummy}
Let $A$, $B$ be two disjoint sets of $\mathbb{Z}$ and let $F$ be a finite dimensional subset in $\mathcal{F}$, then \eqref{Formula2} holds.
\end{lemma}
\begin{proof}
 First, we prove that for an arbitrary but fixed $t'$ we have that
\begin{align}\begin{split}\label{GBT0}
&\lim \limits_{t\rightarrow\infty}\pr(\tau^{A,B}=\infty;\, \forall s\geq t'\,\eta^{A,B}_s\cap(-\infty,0]=\emptyset; \zeta^{A,B}_t\in F)\\
&=\mu_2(F)\pr(\tau^{A,B}=\infty; \,\forall s\geq t'\,\eta^{A,B}_s\cap(-\infty,0]=\emptyset).
\end{split}
\end{align}
To this aim, taking  $t\geq t'$ and using the Markov property we have
\begin{small}\begin{align}\label{GBT1}
\begin{split}
   & \pr(\tau^{A,B}=\infty;\forall s\geq t'\,\eta^{A,B}_s\cap(-\infty,0]=\emptyset; \zeta^{A,B}_t\in F)\\
   &=\int\limits_{\zeta\in \mathcal{C}} \pr(\tau^{\zeta}=\infty; \, \forall\, s\geq 0 \, \eta^{\zeta}_s\cap (-\infty,0]=\emptyset;\zeta^{\zeta}_{t-t'}\in F)\pr(\tau^{A,B}\geq t'|\zeta^{A,B}_{t'}=\zeta)d\nu_{t'}(\zeta),
\end{split}\end{align}\end{small}where $\mathcal{C}$ is the set of configurations that have at least one site occupied by a particle of type $1$ and at least one site occupied by a particle of type $2$. Also, $\nu_{t'}$ is the distribution of $\zeta^{A,B}_{t'}$.

Observe that in the event $\{\forall s\geq 0,\, \eta^{\zeta}_s\cap(-\infty,0]=\emptyset\}$, the process $\zeta^{\zeta}_t$ behaves as the Grass-Bushes-Trees process, in the case where the particles of type $2$ have the priority in all the environment, and the initial configuration is also $\zeta$. We denote this process by $\tilde{\zeta}^{\zeta}_t$. The same ideas used for the classic contact process to obtain \eqref{quintaequacaoahi} hold for the GBT process to obtain
\begin{align}\label{algo}
\begin{split}
&\lim \limits_{t\rightarrow \infty}\pr(\tau^{\zeta}=\infty;\, \forall\, s\geq 0 \, \eta^{\zeta}_s\cap (-\infty,0]=\emptyset;\tilde{\zeta}^{\zeta}_{t-t'}\in F)\\
&=\lim \limits_{t\rightarrow \infty}\pr(\tau^{\zeta}=\infty;\, \forall\, s\geq 0 \, \eta^{\zeta}_s\cap (-\infty,0]=\emptyset;\tilde{\zeta}^{\textbf{1},\textbf{2}}_{t-t'}\in F).
\end{split}
\end{align}
In \cite{Conos} is proved  the tightness of the interface between the particles of type $2$ and the particles of type $1$ for the GBT process with initial configuration $\mathds{1}_{(-\infty,0]}+2\mathds{1}_{[1,\infty)}$. This result implies that the  process $\tilde{\zeta}^{\textbf{1},\textbf{2}}_t$ converges in distribution to $\mu_2$, which together with the limit \eqref{algo} yields  
\begin{align}\label{GBT2}
\begin{split}
&\lim\limits_{t\rightarrow \infty}\pr(\tau^{\zeta}=\infty; \, \forall\, s\geq 0 \, \eta^{\zeta}_s\cap (-\infty,0]=\emptyset;\zeta^{\zeta}_{t-t'}\in F)\\
&=\lim\limits_{t\rightarrow \infty}\pr(\tau^{\zeta}=\infty; \, \forall\, s\geq 0 \, \eta^{\zeta}_s\cap (-\infty,0]=\emptyset;\tilde{\zeta}^{\zeta}_{t-t'}\in F)\\
&=\lim \limits_{t\rightarrow \infty}\pr(\tau^{\zeta}=\infty;\, \forall\, s\geq 0 \, \eta^{\zeta}_s\cap (-\infty,0]=\emptyset;\tilde{\zeta}^{\textbf{1},\textbf{2}}_{t-t'}\in F)\\
&= \mu_2(F)\pr(\tau^{\zeta}=\infty; \, \forall\, s\geq 0 \, \eta^{\zeta}_s\cap (-\infty,0]=\emptyset),
\end{split}
\end{align}
for all $\zeta\in \mathcal{C}$. The last equality in \eqref{GBT2} is a consequence of the convergence in distribution of the process $\tilde{\zeta}^{\textbf{1},\textbf{2}}_t$ to the measure $\mu_2$ and the arguments used to obtain $(2.29)$ in the proof of Theorem $2.28$ page $284$ in \cite{Liggett}. Using the limit \eqref{GBT2}, the Dominated Convergence Theorem, and the Markov property in \eqref{GBT1} we obtain \eqref{GBT0}.

To conclude the lemma, it is sufficient to take $t'$ such that $\pr(\exists \, \bar{t}:\hspace{1mm}\forall s\geq \bar{t}\hspace{1.5mm}\eta^{A,B}_s\cap(-\infty,0]=\emptyset)-\pr(\forall s\geq t'\,\eta^{A,B}_s\cap(-\infty,0]=\emptyset)<\epsilon$ and by a ``$3\epsilon$" argument we have 
\begin{align}\label{lemmadummy3}\begin{split}
    &|\pr(\tau^{A,B}=\infty;\exists \,\bar{t}:\hspace{1mm}\forall s\geq \bar{t}\hspace{1.5mm}\eta^{A,B}_s\cap(-\infty,0]=\emptyset; \zeta^{A,B}_t\in F)\\
   &\quad-\mu_2(F)\pr(\tau^{A,B}=\infty;\exists \, \bar{t}:\forall s\geq \bar{t}\,\eta^{A,B}_s\cap(-\infty,0]=\emptyset)|\\
   &\leq 2\epsilon+|\pr(\tau^{A,B}=\infty;\,\forall s\geq t'\,\eta^{A,B}_s\cap(-\infty,0]=\emptyset; \zeta^{A,B}_t\in F)\\
   &\quad-\mu_2(F)\pr(\tau^{A,B}=\infty;\,\forall s\geq t'\,\eta^{A,B}_s\cap(-\infty,0]=\emptyset)|.
\end{split}
\end{align}
Taking the limit when $t$ goes to infinity on both sides of the inequality \eqref{lemmadummy3}, and then taking $\epsilon$ close to zero, we obtain \eqref{Formula2}.
\end{proof}
\begin{lemma}\label{ultimo}
Let $A$, $B$ be two finite and disjoint sets of $\mathbb{Z}$ and let $F$ be a finite dimensional subset in $\mathcal{F}$, then \eqref{Formula3} holds.
\end{lemma}
\begin{proof}
Consider the  sets $\chi$, $\eta$ and  $D$ defined in \eqref{losD}. Also, define $\textbf{t}$ as the first time such that $\zeta^{A,B}_s\equiv \zeta^{\textbf{1},\textbf{2}}_s\text{ in }E,\,\forall\, s\geq\textbf{t}$. We take $t$ large enough such that
\begin{equation*}
\pr(\exists\, \bar{t}: \, \zeta^{A,B}_s\equiv \zeta^{\textbf{1},\textbf{2}}_s\text{ in }E\,\forall\, s\geq\bar{t};\,t\leq\textbf{t})=\pr(t\leq \textbf{t}<\infty)\leq \epsilon.
\end{equation*}
By Proposition \ref{lallave} and our choice of $t$ we have 
\begin{align*}
\begin{split}
\left|\pr(\zeta^{A,B}_t\in F;D)-\pr(\zeta^{\textbf{1},\textbf{2}}_t\in F;D)\right|&\leq\pr(\zeta^{A,B}_t\not\equiv \zeta^{\textbf{1},\textbf{2}}_t\text{ in }E;D)\\
&\leq \pr(\zeta^{A,B}_t\not\equiv \zeta^{\textbf{1},\textbf{2}}_t\text{ in }E;\exists\, \bar{t}: \, \zeta^{A,B}_s\equiv \zeta^{\textbf{1},\textbf{2}}_s\text{ in }E\,\forall\, s\geq\bar{t})\\
&=\pr(\zeta^{A,B}_t\not\equiv \zeta^{\textbf{1},\textbf{2}}_t\text{ in }E;\,\textbf{t}<\infty)\leq\pr(t\leq\textbf{t}<\infty)\leq \epsilon.
\end{split}
\end{align*}
Then, it is sufficient to prove  
\begin{equation}\label{sidiosquiereyatermino0}
\lim\limits_{t\rightarrow \infty}\pr(\zeta^{\textbf{1},\textbf{2}}_t\in F;D(A,B))=\nu(F)\pr(D(A,B)).
\end{equation}
To obtain \eqref{sidiosquiereyatermino0}  we prove the following limits
\begin{align}
&\lim \limits_{t\rightarrow\infty} \pr(\zeta^{\textbf{1},\textbf{2}}_t\in F;\chi^{c})=\nu(F)\pr(\chi^{c})\label{lim1}\\
&\lim \limits_{t\rightarrow\infty} \pr(\zeta^{\textbf{1},\textbf{2}}_t\in F;\eta^{c})=\nu(F)\pr(\eta^{c})\label{lim2}\\
&\lim \limits_{t\rightarrow\infty} \pr(\zeta^{\textbf{1},\textbf{2}}_t\in F;\chi^{c}\cap \eta^c)=\nu(F)\pr(\chi^{c}\cap \eta^{c}).\label{lim3}
 \end{align}
These limits, together with the fact that $\zeta^{\textbf{1},\textbf{2}}_t$ converges in distribution to $\nu$, imply \eqref{sidiosquiereyatermino0}.
The idea to obtain the limits \eqref{lim1}, \eqref{lim2} and \eqref{lim3} is the same for all of them.  First, we approximate the probability of the event that does not depend on $t$ by the probability of an event that depends on a finite time, and then we use the Markov property. Since the proofs are very similar, we only give the details of the limit \eqref{lim3}.   We take $T$ and $T'$ such that
\begin{align*}
\pr(\chi^{c}\cap \eta^{c})-\pr(\,\eta^{A,B}_s\cap(-\infty,0]=\emptyset,\, \chi^{A,B}_s\cap[1,\infty)=\emptyset\,\forall\,s \in [T, T'])\leq \epsilon.
\end{align*}
To simplify notation, we denote the event in the second probability by $B[T,T']$. We observe that $B[T,T']$ is an event in $\mathcal{F}_{T'}$. Therefore, for $t\geq T'$ we can use the Markov property as follows
\begin{align*}
\pr(\zeta^{\textbf{1},\textbf{2}}_t\in F; B[T,T'])&=\int\pr(\zeta^{\textbf{1},\textbf{2}}_t\in F|\zeta^{\textbf{1},\textbf{2}}_{T'}=\zeta_0)\pr(B[T,T']|\zeta^{\textbf{1},\textbf{2}}_{T'}=\zeta_0)d\nu_{T'}(\zeta_0)\\
&=\int\pr(\zeta^{\zeta_0}_{t-T'}\in F)\pr(B[T,T']|\zeta^{\textbf{1},\textbf{2}}_{T'}=\zeta_0)d\nu_{T'}(\zeta_0)\\
&=\int_{\zeta_0\in \mathcal{A}}\pr(\zeta^{\zeta_0}_{t-T'}\in F)\pr(B[T,T']|\zeta^{\textbf{1},\textbf{2}}_{T'}=\zeta_0)d\nu_{T'}(\zeta_0)
\end{align*}
where $\nu_{T'}$ is the law of $\zeta^{\textbf{1},\textbf{2}}_{T'}$. Taking the limit when $t$ goes to infinity and using the Dominated Convergence Theorem and Corollary \ref{el que pensé que no usaba}, we have
\begin{align*}
\lim\limits_{t\rightarrow\infty}\pr(\zeta^{\textbf{1},\textbf{2}}_t\in F; B[T,T'])&=\int_{\zeta_0\in \mathcal{A}}\nu(F)\pr(B[T,T']|\zeta^{\textbf{1},\textbf{2}}_{T'}=\zeta_0)d\nu_{T'}(\zeta_0)\\
&=\nu(F)\int\pr(B[T,T']|\zeta^{\textbf{1},\textbf{2}}_{T'}=\zeta_0)d\nu_{T'}(\zeta_0)\\
&=\nu(F)\pr(B[T,T']).
\end{align*}
Hence, for $t$ large enough we have
\begin{align}\label{sidiosquiereyatermino1}
\begin{split}
&|\pr(\zeta^{\textbf{1},\textbf{2}}_t\in F;\chi^{c}\cap \eta^{c})-\nu(F)\pr(\chi^{c}\cap \eta^{c}))|\leq |\pr(\zeta^{\textbf{1},\textbf{2}}_t\in F;\chi^{c}\cap \eta^{c})-\pr(\zeta^{\textbf{1},\textbf{2}}_t\in F; B[T,T'])|\\
&+|\pr(\zeta^{\textbf{1},\textbf{2}}_t\in F; B[T,T'])-\nu(F)\pr(B[T,T'])|+|\nu( F)(\pr(B[T,T']-\pr(\chi^{c}\cap \eta^{c}))|\\
&\leq 2\epsilon+|\pr(\zeta^{\textbf{1},\textbf{2}}_t\in F; B[T,T'])-\nu(F)\pr(B[T,T'])|,
\end{split}
\end{align}
where in the last inequality of \eqref{sidiosquiereyatermino1} we have used our choice of $T$ and $T'$. Therefore
\begin{align*}
&\limsup\limits_{t\rightarrow\infty}|\pr(\zeta^{\textbf{1},\textbf{2}}_t\in F;\chi^{c}\cap \eta^{c})-\nu(F)\pr(\chi^{c}\cap \eta^{c})|\leq 2\epsilon.
\end{align*}
\end{proof}

%
%

%
%

 \section*{Acknowledgements}
The author was supported by FAPESP grant $20/02662-4$ post doctoral fellowship. The author thanks  Majela Pent\'on Machado for a careful reading of this work and the many constructive suggestions which improved the exposition considerably. The author also thanks Enrique Andjel and Maria Eulalia Vares for the helpful comments during the preparation of this paper.
\bibliographystyle{acm}
\addcontentsline{toc}{part}{Bibliography}
\bibliography{arxiv.bbl}

\end{document}